\documentclass[12pt]{amsart}
\usepackage{ amsmath, amsthm, amsfonts, amssymb, color}
 \usepackage{mathrsfs}
\usepackage{amsfonts, amsmath}
 \usepackage{amsmath,amstext,amsthm,amssymb,amsxtra}
 \usepackage{txfonts} 
 \usepackage[colorlinks, citecolor=blue,pagebackref,hypertexnames=false]{hyperref}
 \allowdisplaybreaks
 \usepackage{pgf,tikz}
 \usepackage{multirow}
 \usepackage{diagbox} 

 \usepackage{tikz}

\usetikzlibrary{decorations.pathreplacing}

\begin{document}

 \baselineskip 16.6pt
\hfuzz=6pt

\widowpenalty=10000

\newtheorem{cl}{Claim}
\newtheorem{theorem}{Theorem}[section]

\newtheorem{proposition}[theorem]{Proposition}
\newtheorem{corollary}[theorem]{Corollary}
\newtheorem{lemma}[theorem]{Lemma}

\theoremstyle{definition}
\newtheorem{definition}[theorem]{Definition}

\newtheorem{assum}{Assumption}[section]
\newtheorem{example}[theorem]{Example}
\newtheorem{Pro}[theorem]{Problem}
\newtheorem{remark}[theorem]{Remark}
\renewcommand{\theequation}
{\thesection.\arabic{equation}}

\def\SL{\sqrt H}

\newcommand{\mar}[1]{{\marginpar{\sffamily{\scriptsize
        #1}}}}

\newcommand{\as}[1]{{\mar{AS:#1}}}

\newcommand\R{\mathbb{R}}
\newcommand\RR{\mathbb{R}}
\newcommand\CC{\mathbb{C}}
\newcommand\NN{\mathbb{N}}
\newcommand\ZZ{\mathbb{Z}}
\newcommand\HH{\mathbb{H}}
\def\RN {\mathbb{R}^n}
\renewcommand\Re{\operatorname{Re}}
\renewcommand\Im{\operatorname{Im}}

\newcommand{\mc}{\mathcal}
\newcommand\D{\mathcal{D}}
\def\hs{\hspace{0.33cm}}
\newcommand{\la}{\alpha}
\def \l {\alpha}
\newcommand{\eps}{\tau}
\newcommand{\pl}{\partial}
\newcommand{\supp}{{\rm supp}{\hspace{.05cm}}}
\newcommand{\x}{\times}
\newcommand{\lag}{\langle}
\newcommand{\rag}{\rangle}

\newcommand\wrt{\,{\rm d}}

\newtheorem{theoremA}{Theorem}
\renewcommand{\thetheoremA}{\Alph{theoremA}} 

\title[]{Sparse bounds for rough non-translation-invariant Calder\'{o}n-Zygmund operators  }
\maketitle
%
%

\vspace{-0.5cm}
\begin{center}
{\bf Yanping Chen \footnote{\small {Corresponding author.\ }}}\\
Department of Mathematics,\\
Northeastern University,\\
Shenyang 110004, China\\
E-mail: {\it yanpingch@126.com}
\vskip 0.1cm
{\bf Teng Wang}\\
School of Mathematics and Physics,\\
University of Science and Technology Beijing,\\
Beijing 100083,  China \\
E-mail: {\it tengwustb@126.com} \vskip 0.1cm

\end{center}
\vspace{-1cm}

\date{\today}

\renewcommand{\thefootnote}{}



\begin{abstract}
In this paper, we establish sparse bounds for singular integral operators with rough non-translation-invariant Calder\'{o}n-Zygmund kernels.

Our approach verifies the localized hypothesis in the sparse domination framework introduced by Conde-Alonso, Di Plinio et al.
The proof refines Seeger's microlocal decomposition using a Littlewood-Paley analysis adapted to the non-translation-invariant setting.

As applications, we obtain quantitative weighted estimates for Calder\'{o}n-type commutators and other singular integrals with rough kernels.

\end{abstract}
	
	\maketitle

\arraycolsep=1pt

\section{Introduction}

Quantitative weighted estimates for singular integral operators have been extensively studied in the past decade.
A central theme in this area is to determine the \emph{sharp dependence} of operator norms on the Muckenhoupt $A_p$ characteristics.
For classical Calder\'{o}n-Zygmund operators,  the sharp weighted bounds are now well understood; see, for instance, \cite{Hytonen, Lacey, Lerner, LN13}.

For singular integrals with rough kernels, Hyt\"{o}nen, Roncal and Tapiola \cite{HRT} obtained the first quantitative estimate for the rough homogeneous singular integral
\begin{equation}\label{IOd}
  I_{\Omega}f(x)=\text{p.v.} \int_{\mathbb{R}^n}\frac{\Omega(x-y)}{|x-y|^n}
f(y)dy,
\end{equation}
where $\Omega\in L^{1}(\mathbb{S}^{n-1})$ is a homogeneous function of degree zero:
\begin{equation}\label{a}
  \Omega(rx')=\Omega(x'), \quad \text{for  any } x'\in \mathbb{S}^{n-1}\text{ and } r>0,
\end{equation}
and satisfies the cancellation condition
\begin{equation}\label{c}
\int_{\mathbb{S}^{n-1}}\Omega(x') d\sigma(x')=0.
\end{equation}
In particular, for $w\in A_2$ (see Subsection \ref{weights} for definitions), they proved
\begin{equation*}
\| I_{\Omega}f\|_{L^2(w)} \lesssim [w]_{A_2}^{2}\|f\|_{L^2(w)}.
\end{equation*}

Significant progress has been made toward understanding the quantitative weighted bounds of singular integrals with rough kernels; see, e.g., \cite{CLRT2021, Chen, ChenLi, CYZ, CCDO, DHL2020, GraW, LPRR}.

A major advance in quantitative bounds is the development of \emph{sparse domination} techniques.
These methods allow singular integral operators to be dominated by positive sparse forms that take the form of discrete sums over dyadic cubes satisfying strong sparsity conditions.
For a systematic account, see \cite{Perey}.

One common approach to sparse domination relies on weak $L^p$ bounds for the associated grand maximal operators; see, e.g., \cite{BFP, HRT, Lacey, LN13, Lerner16, Li18}.
Nevertheless, as observed in \cite{Lerner16}, it is not even clear whether the grand maximal operator is bounded on $L^2$ for $I_{\Omega}$, even $\Omega\in L^{\infty}(\mathbb{S}^{n-1})$.
In \cite{CCDO}, a new approach was developed to obtain sparse domination results, relying only on the uniform $L^2$ (or $L^r$ for any $r>1$) boundedness of the truncated operators.
This method has proven to be a powerful and flexible tool in the study of quantitative weighted bounds beyond the classical Calder\'{o}n-Zygmund theory.
There are many other recent developments that carry the sparse domination techniques much further; see, e.g., \cite{Ball, Belt, Culiuc}.
Despite these advances, sparse domination for non-convolution operators with rough kernels has not yet been established.

Our original motivation comes from establishing sparse domination of the Calder\'{o}n commutator with rough kernels, introduced by Calder\'{o}n \cite{Ca65}, defined by
\begin{equation}\label{TOb}
  T_{\Omega, B}f(x)=\text{p.v.} \int_{\mathbb{R}^n}\frac{\Omega(x-y)}{|x-y|^n}\frac{B(x)-B(y)}{|x-y|}
f(y)dy,
\end{equation}
where $\Omega\in L^1(\mathbb{S}^{n-1})$ satisfies \eqref{a}, $B\in \mathrm{Lip}(\mathbb{R}^n)$.
This operator is challenging, as its kernel is both rough and non-convolution.
Consequently, the approach in \cite{CCDO} cannot be applied directly, since verifying the localized hypothesis (L) (see Lemma \ref{thm:C}) for \(I_\Omega\) relies on its convolution structure.
To overcome these obstacles, we introduce a more general class of singular integral operators that provides a unified setting encompassing both the rough Calder\'{o}n commutator and other important models.


In this paper, we consider singular integral operators of the form
\begin{equation}\label{TO}
T_{\Omega} f(x) = \text{p.v.} \int_{\mathbb{R}^{n}} \Omega(x-y) K(x,y) f(y)\, dy,
\end{equation}
where $\Omega\in L^{1}(\mathbb{S}^{n-1})$ satisfies \eqref{a}
and $K$ satisfies the size condition
\begin{equation}\label{Ksize}
  |K(x,y)| \leq \frac{C}{|x-y|^n},
\end{equation}
and the regularity conditions: for some fixed $\delta\in(0,1]$,
\begin{equation}\label{Ksmooth}
  \begin{aligned}
    |K(x_1,y)-K(x_2,y)|\leq C\frac{|x_1-x_2|^{\delta}}{|x_1-y|^{n+\delta}},\quad |x_1-y|>2|x_1-x_2|,\\
    |K(x,y_1)-K(x,y_2)|\leq C\frac{|y_1-y_2|^{\delta}}{|x-y_1|^{n+\delta}},\quad |x-y_1|>2|y_1-y_2|.
  \end{aligned}
\end{equation}
In addition, we assume that $\Omega$ and $K$ satisfy appropriate cancellation conditions, so that $T_{\Omega}f(x)$ is well defined for $f\in C^{\infty}_c(\mathbb{R}^n)$, and can be extended to a bounded operator on $L^2(\mathbb{R}^n)$.

This class of singular integral operators $T_{\Omega}$ was first introduced by Ding and Lai \cite{DL} in their study of weak type $(1,1)$ boundedness criteria for singular integrals with rough kernels.
It encompasses several important models.
When $\Omega$ is a constant, $T_{\Omega}$ reduces to a standard Calder\'{o}n-Zygmund operator.
If $K(x,y)=|x-y|^{-n}$ and $\Omega$ satisfies the cancellation condition \eqref{c},
then $T_{\Omega}$ recovers the rough homogeneous singular integral $I_{\Omega}$.
Moreover, \(T_\Omega\) covers a variety of important non-convolution operators, including the Calder\'{o}n commutator \cite{Ca65}, its higher-order \cite{Hof90}, fractional \cite{Sego71}, and general variants \cite{Ca80}, the Baj\v{s}anski-Coifman type commutator \cite{Bajsan}, and general Muckenhoupt-type singular integrals \cite{Mucken}.

This unified framework allows us to handle a broad class of singular integral operators and systematically analyze their sparse bounds and quantitative weighted behavior.
Our approach is in the spirit of \cite{CCDO, DL, Seeger2},
building on Lemma \ref{thm:C}, we refine Seeger's microlocal decomposition by incorporating a Littlewood-Paley analysis, adapting it to the non-translation-invariant setting, and use it to verify the localized hypothesis \eqref{L}.
As a result, we establish new sparse bounds and quantitative weighted estimates for rough singular integrals with non-translation-invariant Calder\'{o}n-Zygmund kernels. While these estimates are analogous to those for convolution operators, their extension to non-convolution operators provides new results that were previously unavailable.

\subsection{Main Result}
To state our main result precisely, recall that a collection $\mathcal{S}$ contained in a dyadic lattice $\mathcal{D}$ is called an $\eta$-sparse $(0<\eta<1)$ family if for each $Q \in \mathcal{S}$ there exists a measurable set $E_Q\subset Q$ such that $\eta|Q| \leq |E_Q|$, and the sets $\{E_Q\}_{Q\in\mathcal{S}}$ are pairwise disjoint. For a detailed account of dyadic lattices and sparse families, see \cite{LN}.

We also recall the Orlicz-Lorentz norm $L^{q,1}\log L(\mathbb{S}^{n-1})$, defined by
$$
\|\Omega\|_{L^{q,1}\log L} := q \int_0^{\infty} t \log(e+t)
|\{\theta \in \mathbb{S}^{n-1} : |\Omega(\theta)| > t\}|^{\frac{1}{q}} \frac{dt}{t}, \quad 1 \le q < \infty.
$$

Our main result can be stated as follows.

\begin{theorem} \label{Thm}
Let $1<q\leq \infty$, $\Omega\in L^q(\mathbb{S}^{n-1})$ satisfy \eqref{a}.
Suppose $K$ satisfies \eqref{Ksize} and \eqref{Ksmooth}.
In addition, assume that $\Omega$ and $K$ satisfy suitable cancellation conditions such that $T_{\Omega}f(x)$ in \eqref{TO} is well defined for $f\in C^{\infty}_c(\mathbb{R}^n)$ and extends to a bounded operator on $L^2(\mathbb{R}^n)$, with operator norm controlled by $\|\Omega\|_{L^q}$.
Then, for all $1 < s < \infty$, $f \in L^s(\mathbb{R}^{n})$, and $g \in L^{s'}(\mathbb{R}^{n})$, there holds
\begin{align*}
\Big|\int_{\mathbb{R}^{n}} T_{\Omega}(f) g \, dx \Big| \lesssim p' \sup_{\mathcal{S}} \sum_{Q \in \mathcal{S}}
\Big( \int_Q |f| \Big) \Big( \frac{1}{|Q|} \int_Q |g|^p \Big)^{1/p}
\begin{cases}
\|\Omega\|_{L^{q,1}\log L}, & 1<q<\infty, p \geq q',\\
\|\Omega\|_{L^\infty}, & 1<p<\infty.
\end{cases}
\end{align*}
\end{theorem}

\begin{remark}\label{remark1}
Theorem \ref{Thm} provides a unified sparse domination criterion for a wide class of singular integral operators with rough non-convolution kernels, including the Calder\'{o}n commutator, its higher-order, fractional, and general variants, the Baj\v{s}anski-Coifman commutator, and general Muckenhoupt-type singular integrals.
In the special case where \(K(x,y)=|x-y|^{-n}\) and \(\Omega\in L^{q}(\mathbb{S}^{n-1})\) with \(1<q\leq\infty\), satisfying \eqref{a} and \eqref{c}, \(T_{\Omega}\) reduces to the rough homogeneous singular integral \(I_{\Omega}\), which is bounded on \(L^2(\mathbb{R}^n)\) with norm controlled by \(\|\Omega\|_{L^q}\).
In this case, our sparse domination criterion recovers the result of \cite{CCDO}.
\end{remark}

We next highlight the main differences between our approach and previous sparse domination schemes:

\begin{enumerate}
  \item[\;(i)]
    The verification of assumption \eqref{L} in \cite{CCDO} relies on a decay estimate that uses the commutativity of convolution operators, which fails in the non-convolution setting. We overcome this by employing a microlocal analysis technique developed in \cite{DL, Seeger2}.

 \item[\;(ii)]
    While Ding and Lai \cite{DL} use the Mihlin multiplier and pigeonholing in certain steps, this is not suitable for our sparse decomposition.
    We instead apply an $L^1$ Fourier multiplier estimate inside the microlocal decomposition, following \cite{Seeger2} to exploit the cancellation of $I-P_{j-k+\kappa_{\delta}(k)}$ and control the remainder terms efficiently.

 \item[\;(iii)]
  Following \cite{DL}, we also note that the cancellation of $\Omega$ is not essential for our sparse domination, provided $T_\Omega$ is $L^2$ bounded.

\end{enumerate}

\smallskip
The strategy in the proof of Theorem \ref{Thm} is as follows.
We first apply a partition of unity to decompose the kernel.
Next, by using the sparse domination theory introduced by Conde-Alonso, Di Plinio et al. (see Lemma \ref{thm:C}), we reduce the problem to verifying the assumption \eqref{L} (cf. \eqref{forL}).
We then split the index \(k\) into two regimes.
For large \(k\), we approximate \(T_j\) by an operator \(T^k_j\) obtained by regularizing the kernel at the appropriate scale,
and decompose
$$T^k_j=P_{j-k+\kappa_{\delta}(k)}T^k_j+(I-P_{j-k+\kappa_{\delta}(k)})T^k_j.$$
The first term requires a refined use of the cancellation properties of $b_L$ (see Lemma \ref{Pbh}).
For the second term, we further split the kernel $\Omega$ into its bounded and unbounded parts,
and the main difficulty lies in obtaining the required decay for the bounded part.
This is achieved by a microlocal decomposition, combined with Littlewood-Paley decomposition and $L^1$ Fourier multiplier estimates (see Subsection \ref{UV}).

\medskip
As a direct consequence of sparse domination criterion, we derive new quantitative weighted bounds for $T_{\Omega}$.
Theorem \ref{Thm} implies that $T_{\Omega}$ enjoys the same sparse domination structure as the rough homogeneous singular integral $I_\Omega$ studied in \cite{CCDO}.
Moreover, since $T_{\Omega}$ is essentially self-dual, its dual operator $T_{\Omega}^t$ has kernel $\widetilde{K_{\Omega}}(x,y) = \Omega(y-x) K(y,x)$, the same sparse domination applies to $T_{\Omega}^t$.
Consequently, the weighted estimates in \cite{LPRR} follow without further modification, yielding the following quantitative weighted bounds.

\begin{corollary}\label{cor1}
Let $1<q\leq \infty$ and $\Omega\in L^q(\mathbb{S}^{n-1})$ satisfy \eqref{a}.
Suppose $K$ satisfies \eqref{Ksize} and \eqref{Ksmooth}.
In addition, assume that $\Omega$ and $K$ satisfy suitable cancellation conditions such that $T_{\Omega}f(x)$ in \eqref{TO} is well defined for $f\in C^{\infty}_c(\mathbb{R}^n)$ and extends to a bounded operator on $L^2(\mathbb{R}^n)$, with operator norm controlled by $\|\Omega\|_{L^q}$.
Then $T_{\Omega}$ admits the following quantitative weighted bounds:
\begin{enumerate}
\item (\textbf{$A_p$ bounds}) For any $1 < p < \infty$ and $w\in A_p$,
\begin{equation}\label{Cor1}
\| T_{\Omega}f\|_{L^p(w)} \lesssim [w]_{A_p}^{\frac{1}{p}}([w]_{A_{\infty}}^{\frac{1}{p'}}+[w^{\frac{1}{1-p}}]_{A_{\infty}}^{\frac{1}{p}})\min\{[w]_{A_{\infty}},[w^{\frac{1}{1-p}}]_{A_\infty}\}\|\Omega\|_{L^\infty}\|f\|_{L^p(w)},
\end{equation}
with the corresponding weak type estimate
\begin{align*}
  \|T_{\Omega}f\|_{L^{p,\infty}(w)}\lesssim [w]_{A_p}^{\min\{2, p'\}}\|\Omega\|_{L^\infty} \|f\|_{L^p(w)}.
\end{align*}

\item (\textbf{$A_1$ weak type bound}) For $p=1$ and $w\in A_1$,
\begin{equation}\label{wCor1}
  \|T_{\Omega}f\|_{L^{1,\infty}(w)}\lesssim [w]_{A_1} [w]_{A_\infty} \log_2([w]_{A_\infty}+1)\|\Omega\|_{L^\infty}\|f\|_{L^1(w)}.
\end{equation}

\item (\textbf{Sharp $A_s$ bounds}) For all $1 \leq s < p$ and $w\in A_s$,
\begin{equation}\label{cor2s}
\begin{aligned}
\| T_{\Omega}f\|_{L^p(w)} \lesssim [w]_{A_s}\|f\|_{L^p(w)}
\begin{cases}
\|\Omega\|_{L^{q,1}\log L}, & 1<q<\infty, p >q',\\
\|\Omega\|_{L^\infty}, & 1<p<\infty.
\end{cases}
\end{aligned}
\end{equation}
\end{enumerate}
\end{corollary}

\begin{remark}\label{remark2}
Corollary \ref{cor1} recovers the quantitative weighted bounds for $I_\Omega$ established in \cite{LPRR} and applies to the broad class of non-convolution singular integrals described in Remark \ref{remark1}.
\end{remark}

We next discuss the main consequences of Corollary \ref{cor1} in relation to previously known results.

Estimate \eqref{Cor1} corresponds to the improved $A_p$ bound introduced in \cite{LPRR}.
More precisely, the right-hand side of \eqref{Cor1} can be compared with the usual $A_p$ exponent as
$$[w]_{A_p}^{\frac{1}{p}}([w]_{A_{\infty}}^{\frac{1}{p'}}+[w^{\frac{1}{1-p}}]_{A_{\infty}}^{\frac{1}{p}})\min\{[w]_{A_{\infty}},[w^{\frac{1}{1-p}}]_{A_\infty}\}\lesssim [w]_{A_p}^{{p'}}.
$$
It is easy to see that $$p'\leq \frac{1}{p-1}\max\{2,p\}\leq 2\max\{1,\frac{1}{p-1}\}.$$
Here, the second exponent was obtained for the rough homogeneous singular integral $I_{\Omega}$ in \cite{CCDO} and the third in \cite{HRT}.
Whether the exponent of the Muckenhoupt constant in \eqref{Cor1} is sharp remains an open question; nevertheless, the present estimates represent the best known bounds obtainable by existing techniques.

The weak type estimate \eqref{wCor1} is related to the weak Muckenhoupt and Wheeden conjecture \cite{Muckenh}.
For the standard Calder\'{o}n-Zygmund operator, the corresponding sharp bound was established in \cite{Ler09}, with a discussion on sharpness in \cite{LNO20}.
For the rough homogeneous singular integral $I_{\Omega}$, see \cite{LPRR}.
The additional factor $[w]_{A_{\infty}}$ in \eqref{wCor1} reflects the roughness of the kernel, and its sharpness is currently unknown.

Estimate \eqref{cor2s} provides sharp bounds for $T_{\Omega}$ on $L^p(w)$ for weights in $A_s$, $1\le s <p$.
For $s=1$, this bound was first obtained for standard Calder\'{o}n-Zygmund operators in \cite{Ler08}, and for $1<s<p$ via extrapolation in \cite{Duo11}.
For $I_{\Omega}$, see \cite{LPRR}.
A discussion on the sharpness of the $[w]_{A_s}$ exponent is given in \cite{Ler12}.
Further background and related results may be found in \cite{CoiF, CorF, Duo11, Feff97, LNO20, Ler12, Ler09}.

\subsection{Applications}
As mentioned in Remark \ref{remark2}, Corollary \ref{cor1} applies to several important classes of non-convolution singular integrals.

We illustrate our results with the following examples:
\begin{itemize}
  \item the Calder\'{o}n commutator, together with its higher-order variants,
  \item the fractional Calder\'{o}n commutator.
  \item the general Calder\'{o}n commutator,
  \item the Baj\v{s}anski-Coifman type Calder\'{o}n commutator,
  \item the general Muckenhoupt-type singular integrals.
\end{itemize}
For all of these operators, we obtain new quantitative weighted bounds, weak $(1,1)$ estimates for $A_1$ weights, and sharp $[w]_{A_s} (1\leq s<p)$ dependence.
Notably, these results are new even for the Calder\'{o}n commutator.

\subsubsection{Calder\'{o}n Commutator and Its Higher-Order Variants}

As a first application of Corollary \ref{cor1}, we establish the quantitative weighted estimates for the Calder\'{o}n commutator and its higher-order variants.
In 1990, Hofmann \cite{Hof90} proved a weighted norm inequality for the higher-order Calder\'{o}n commutator, defined by
\begin{equation*}
T_{\Omega, B, k}f(x)=\text{p.v.} \int_{\mathbb{R}^n}\frac{\Omega(x-y)}{|x-y|^n}\Big(\frac{B(x)-B(y)}{|x-y|}\Big)^k f(y)dy,
\end{equation*}
where $k\geq 1$, $B\in \mathrm{Lip}(\mathbb{R}^n)$, $\Omega\in L^1(\mathbb{S}^{n-1})$ satisfies \eqref{a} and
\begin{equation}\label{ck}
\int_{\mathbb{S}^{n-1}} \Omega(\theta)\theta^{\alpha} d\theta=0, \text{ for all } \alpha \in \mathbb{Z}^n_+ \text{ with } |\alpha|=k.
\end{equation}
Here, $\alpha=(\alpha_1,\cdots,\alpha_n)\in \mathbb{Z}^n_+$ denotes a multi-index, $|\alpha|=\sum_{j=1}^n\alpha_j$, and $x^{\alpha}=\prod_{i=1}^n x_i^{\alpha_i}$ for $x\in \mathbb{R}^n$.

The case $k=1$ corresponds to the Calder\'{o}n commutator $T_{\Omega, B}$ \eqref{TOb},
which originates from the representation of linear differential operators in terms of singular integrals and provides a tool in studying the uniqueness of the Cauchy problems for partial differential equations (see \cite{Ca58, CaCa78}).
It also plays an important role in the theory of Cauchy integrals along Lipschitz curves in $\mathbb{C}$ and in the Kato square root problem on $\mathbb{R}$ (see \cite{Ca58, Meyer90, Meyer91} for details).

\begin{theorem}\label{wTbk}
Let $k\geq 1$ and $B\in \mathrm{Lip}(\mathbb{R}^n)$. Let $1<q\leq \infty$ and $\Omega\in L^{q}(\mathbb{S}^{n-1})$ satisfy \eqref{a} and \eqref{ck}.
Then the following hold:
\begin{enumerate}
\item For any $1 < p < \infty$ and $w\in A_p$, we have
\begin{equation}\label{wTObkf}
\begin{aligned}
\| T_{\Omega, B, k}f\|_{L^p(w)} \lesssim \|\nabla B\|^{k}_{\infty}  [w]_{A_p}^{\frac{1}{p}}([w]_{A_{\infty}}^{\frac{1}{p'}}+[w^{\frac{1}{1-p}}]_{A_{\infty}}^{\frac{1}{p}})\min\{[w]_{A_{\infty}},[w^{\frac{1}{1-p}}]_{A_\infty}\}
\|\Omega\|_{L^\infty}\|f\|_{L^p(w)}.
\end{aligned}
\end{equation}
Moreover, the following weak type estimate holds:
\begin{align*}
  \|T_{\Omega, B, k}f\|_{L^{p,\infty}(w)}\lesssim \|\nabla B\|^{k}_{\infty}  [w]_{A_p}^{\min\{2, {p'}\}}
   \|\Omega\|_{L^\infty} \|f\|_{L^p(w)}.
\end{align*}
\item  For $p=1$ and $w\in A_1$, we have
\begin{align*}
  \|T_{\Omega, B, k}f\|_{L^{1,\infty}(w)}\lesssim \|\nabla B\|^{k}_{\infty}  [w]_{A_1} [w]_{A_\infty} \log_2([w]_{A_\infty}+1)
   \|\Omega\|_{L^\infty}\|f\|_{L^1(w)}.
\end{align*}
\item For all $1 \leq s < p$ and $w\in A_s$, we have
\begin{align*}
\| T_{\Omega, B, k}f\|_{L^p(w)} \lesssim \|\nabla B\|^{k}_{\infty}   [w]_{A_s}  \|f\|_{L^p(w)}
\begin{cases}
\|\Omega\|_{L^{q,1}\log L}, & 1<q<\infty, p \geq q',\\
\|\Omega\|_{L^\infty}, & 1<p<\infty.
\end{cases}
\end{align*}
\end{enumerate}
\end{theorem}

\begin{remark}
For the special case $k=1$, Theorem \ref{wTbk} provides a quantitative improvement over the bound in \cite{ChenLi}, under a strictly weaker cancellation condition on $\Omega$: notably, the vanishing of the zeroth moment $|\alpha|=0$ is not required.
\end{remark}

\subsubsection{Fractional Calder\'{o}n Commutator}

We first recall the definition of the \emph{fractional differentiation operator} $D^{\gamma}$ of order $\gamma$ ($0<\gamma<1$), which is defined via the Fourier transform as
\[
\widehat{D^\gamma f}(\xi) = |\xi|^\gamma \hat{f}(\xi), \quad \xi \in \mathbb{R}^n.
\]
Similarly, the \emph{Riesz potential operator} $I_\gamma$ of order $\gamma$ is given by
\[
\widehat{I_\gamma f}(\xi) = |\xi|^{-\gamma} \hat{f}(\xi).
\]
These operators are inverses of each other in the sense that $D^\gamma I_\gamma f = f$ for all $f \in \mathcal{S}(\mathbb{R}^n)$, the space of Schwartz functions.

We next introduce the BMO-Sobolev spaces.
\begin{definition}[BMO-Sobolev space]
For $0<\gamma<1$, the \emph{BMO-Sobolev space} $I_\gamma(\mathrm{BMO})(\mathbb{R}^n)$ is defined by
\[
I_\gamma(\mathrm{BMO})(\mathbb{R}^n) := \{\, f = I_\gamma g : g \in \mathrm{BMO}(\mathbb{R}^n) \,\}.
\]
Equivalently, a locally integrable function $f$ belongs to $I_\gamma(\mathrm{BMO})(\mathbb{R}^n)$ if and only if $D^\gamma f \in \mathrm{BMO}(\mathbb{R}^n)$.
\end{definition}

The fractional Calder\'{o}n commutator is defined by
\begin{equation*}
  T_{\Omega, B, \gamma}f(x)=\text{p.v.} \int_{\mathbb{R}^n}\frac{\Omega(x-y)}{|x-y|^n} \frac{B(x)-B(y)}{|x-y|^\gamma}
f(y)dy,
\end{equation*}
where $0<\gamma<1$, $\Omega$ satisfies \eqref{a}, $B\in I_{\gamma}(\mathrm{BMO})(\mathbb{R}^n)$.
This operator was first investigated by Segovia and Wheeden \cite{Sego71}, who obtained an analogue
for differentiation of a fractional order of the one-dimensional version of Calder\'{o}n's result \cite{Ca65}. For more details on the operator $T_{\Omega, B, \gamma}$, we refer the reader to \cite{CDH16, Mur85} and the reference therein.

Now we state our result for $T_{\Omega, B, \gamma}$.

\begin{theorem}\label{Tbr}
Let $\gamma\in (0,1)$ and $B\in I_{\gamma}(\mathrm{BMO})(\mathbb{R}^n)$.
Let $1<q\leq\infty$ and $\Omega\in L^q(\mathbb{S}^{n-1})$ satisfy \eqref{a} and \eqref{c}.
Then the following hold:
\begin{enumerate}
\item For any $1 < p < \infty$ and $w\in A_p$, we have
\begin{align*}
\| T_{\Omega, B,\gamma}f\|_{L^p(w)} \lesssim \|D^{\gamma}B\|_{\mathrm{BMO}} [w]_{A_p}^{\frac{1}{p}}([w]_{A_{\infty}}^{\frac{1}{p'}}+[w^{\frac{1}{1-p}}]_{A_{\infty}}^{\frac{1}{p}})\min\{[w]_{A_{\infty}},[w^{\frac{1}{1-p}}]_{A_\infty}\}
\|\Omega\|_{L^\infty}\|f\|_{L^p(w)}.
\end{align*}
Moreover, the following weak type estimate holds:
\begin{align*}
  \|T_{\Omega,B,\gamma}f\|_{L^{p,\infty}(w)}\lesssim \|D^{\gamma}B\|_{\mathrm{BMO}}  [w]_{A_p}^{\min\{2, {p'}\}}
  \|\Omega\|_{L^\infty} \|f\|_{L^p(w)}.
\end{align*}
\item For $p=1$ and $w\in A_1$, we have
\begin{align*}
  \|T_{\Omega,B,\gamma}f\|_{L^{1,\infty}(w)}\lesssim  \|D^{\gamma}B\|_{\mathrm{BMO}} [w]_{A_1} [w]_{A_\infty} \log_2([w]_{A_\infty}+1)
  \|\Omega\|_{L^\infty}\|f\|_{L^1(w)}.
\end{align*}
\item For all $1 \leq s < p$ and $w\in A_s$, we have
\begin{align*}
\| T_{\Omega,B,\gamma}f\|_{L^p(w)} \lesssim \|D^{\gamma}B\|_{\mathrm{BMO}} [w]_{A_s}  \|f\|_{L^p(w)}
\begin{cases}
\|\Omega\|_{L^{q,1}\log L}, & 1<q<\infty, p >q',\\
\|\Omega\|_{L^\infty}, & 1<p<\infty.
\end{cases}
\end{align*}
\end{enumerate}
\end{theorem}

\begin{remark}It is worth noting that our weighted bounds improve those in \cite[Theorem 1.1]{CYZ}
by providing a smaller exponent in the $A_p$ characteristic of the weight.

\end{remark}

\subsubsection{General Calder\'{o}n Commutator}
In 1980, Calder\'{o}n \cite{Ca80} introduced the general commutator
\begin{equation*}
T_{\Omega, B, F}f(x)=\text{p.v.} \int_{\mathbb{R}^n}\frac{\Omega(x-y)}{|x-y|^n}F\Big(\frac{B(x)-B(y)}{|x-y|}\Big)
f(y)dy.
\end{equation*}

The study of this commutator is closely related to Cauchy integral on Lipschitz curves and boundary value problems for elliptic equations on non-smooth domains (see \cite{Ca77, Ca80, CaCa78, Fabes}).

\begin{theorem}\label{bF1}
Let $1<q\leq \infty$, $\Omega \in L^q(\mathbb{S}^{n-1})$ satisfy $\Omega(-\theta)=-\Omega(\theta)$ for  $\theta \in \mathbb{S}^{n-1}$.
Let $B\in \mathrm{Lip}(\mathbb{R}^n)$, and let $F(t)$ be real analytic in $\{|t|\leq \|\nabla B\|_{\infty}\}$ satisfy $F(t)=F(-t)$ for $t\in \mathbb{R}$.
Then the following hold:
\begin{enumerate}
\item For any $1 < p < \infty$ and $w\in A_p$, we have
\begin{align*}
\| T_{\Omega, B, F} f\|_{L^p(w)} \lesssim C_{B,F} [w]_{A_p}^{\frac{1}{p}}([w]_{A_{\infty}}^{\frac{1}{p'}}+[w^{\frac{1}{1-p}}]_{A_{\infty}}^{\frac{1}{p}})\min\{[w]_{A_{\infty}},[w^{\frac{1}{1-p}}]_{A_\infty}\}
\|\Omega\|_{L^\infty}\|f\|_{L^p(w)}.
\end{align*}
Moreover, the following weak type estimate holds:
\begin{align*}
  \|T_{\Omega, B, F} f\|_{L^{p,\infty}(w)}\lesssim C_{B,F} [w]_{A_p}^{\min\{2, {p'}\}}
   \|\Omega\|_{L^\infty} \|f\|_{L^p(w)}.
\end{align*}
\item For $p=1$ and $w\in A_1$, we have
\begin{align*}
  \|T_{\Omega, B, F}f\|_{L^{1,\infty}(w)}\lesssim  C_{B,F} [w]_{A_1} [w]_{A_\infty} \log_2([w]_{A_\infty}+1)
   \|\Omega\|_{L^\infty}\|f\|_{L^1(w)}.
\end{align*}
\item For all $1 \leq s < p$ and $w\in A_s$, we have
\begin{align*}
\| T_{\Omega, B, F} f\|_{L^p(w)} \lesssim C_{B,F} [w]_{A_s} \|f\|_{L^p(w)}
\begin{cases}
\|\Omega\|_{L^{q,1}\log L}, & 1<q<\infty, p >q',\\
\|\Omega\|_{L^\infty}, & 1<p<\infty,
\end{cases}
\end{align*}
\end{enumerate}
where $C_{B,F}$ depends only on $n$, $B$, and $F$.
\end{theorem}

\subsubsection{Baj\v{s}anski-Coifman Type Calder\'{o}n Commutator}

In 1967, Baj\v{s}anski and Coifman \cite{Bajsan} introduced a higher-order analogue of the Calder\'{o}n commutator.
For a multi-index $\alpha\in \mathbb{Z}^n_+$, denote $B_{\alpha}(x)=\partial^{\alpha}_x B(x)$ and define
$$
P_l(B,x,y)=B(x)-\sum_{|\alpha|<l}\frac{B_{\alpha}(y)}{\alpha !}(x-y)^{\alpha},
$$
where $l\in \mathbb{N}$.
The singular integral operator $T_{\Omega,B}^{l}$ is then defined by
\begin{equation*}
  T_{\Omega, B}^{l}f(x)=\text{p.v.} \int_{\mathbb{R}^n}\frac{\Omega(x-y)}{|x-y|^n}\frac{P_l(B,x,y)}{|x-y|^l}
f(y)dy,
\end{equation*}
where $\Omega\in L^1(\mathbb{S}^{n-1})$ satisfies \eqref{a} and
\begin{equation}\label{cl}
\int_{\mathbb{S}^{n-1}} \Omega(\theta)\theta^{\alpha} d\theta=0, \text{ for all } \alpha \in \mathbb{Z}^n_+ \text{ with } |\alpha|=l.
\end{equation}

\begin{theorem}\label{Tbpl}
Let $l\in \mathbb{N}$ and assume $B_{\alpha}\in L^{\infty}(\mathbb{R}^n)$ for all $|\alpha|=l$.
Let $1<q\leq \infty$ and $\Omega\in L^q(\mathbb{S}^{n-1})$ satisfy \eqref{a} and \eqref{cl}.
Then the following hold:
\begin{enumerate}
\item For any $1 < p < \infty$ and $w\in A_p$, we have
\begin{align*}
\| T_{\Omega, B}^lf\|_{L^p(w)} \lesssim \sum_{|\alpha|=l} \|B_{\alpha}\|_{\infty} [w]_{A_p}^{\frac{1}{p}}([w]_{A_{\infty}}^{\frac{1}{p'}}+[w^{\frac{1}{1-p}}]_{A_{\infty}}^{\frac{1}{p}})\min\{[w]_{A_{\infty}},[w^{\frac{1}{1-p}}]_{A_\infty}\}
\|\Omega\|_{L^\infty}\|f\|_{L^p(w)}.
\end{align*}
Moreover, the following weak type estimate holds:
\begin{align*}
  \|T_{\Omega, B}^lf\|_{L^{p,\infty}(w)}\lesssim \sum_{|\alpha|=l} \|B_{\alpha}\|_{\infty} [w]_{A_p}^{\min\{2, {p'}\}}
  \|\Omega\|_{L^\infty} \|f\|_{L^p(w)}.
\end{align*}
\item For $p=1$ and $w\in A_1$, we have
\begin{align*}
  \|T_{\Omega, B}^lf\|_{L^{1,\infty}(w)}\lesssim \sum_{|\alpha|=l} \|B_{\alpha}\|_{\infty} [w]_{A_1} [w]_{A_\infty} \log_2([w]_{A_\infty}+1)
  \|\Omega\|_{L^\infty}\|f\|_{L^1(w)}.
\end{align*}
\item For all $1 \leq s < p$ and $w\in A_s$, we have
\begin{align*}
\| T_{\Omega, B}^lf\|_{L^p(w)} \lesssim \sum_{|\alpha|=l} \|B_{\alpha}\|_{\infty} [w]_{A_s}\|f\|_{L^p(w)}
\begin{cases}
\|\Omega\|_{L^{q,1}\log L}, & 1<q<\infty, p >q',\\
\|\Omega\|_{L^\infty}, & 1<p<\infty.
\end{cases}
\end{align*}
\end{enumerate}
\end{theorem}

\subsubsection{Muckenhoupt Type General Singular Integral}
In 1960, Muckenhoupt \cite{Mucken} considered a modification of singular integrals, generalizing the earlier work of Calder\'{o}n and Zygmund \cite{CZ52, CZ56} on fractional integrals.

Assume that $\Omega\in L^1(\mathbb{S}^{n-1})$ satisfies \eqref{a} and \eqref{c}.
For $r\in \mathbb{R}\backslash\{0\}$, define
\begin{equation*}
T_{\Omega, ir}f(x)=\text{p.v.} \int_{\mathbb{R}^n} \frac{\Omega(x-y)}{|x-y|^{n+i r}} f(y)dy,
\end{equation*}
where $i=\sqrt{-1}$.

\begin{theorem}\label{ir1}
Let $\Omega\in L^1(\mathbb{S}^{n-1})$ satisfy \eqref{a} and \eqref{c}.
Then the following hold:
\begin{enumerate}
\item For any $1 < p < \infty$ and $w\in A_p$, we have
\begin{align*}
\| T_{\Omega, i r}f\|_{L^p(w)} \lesssim [w]_{A_p}^{\frac{1}{p}}([w]_{A_{\infty}}^{\frac{1}{p'}}+[w^{\frac{1}{1-p}}]_{A_{\infty}}^{\frac{1}{p}})\min\{[w]_{A_{\infty}},[w^{\frac{1}{1-p}}]_{A_\infty}\}\|\Omega\|_{L^\infty}\|f\|_{L^p(w)}.
\end{align*}
Moreover, the following weak type estimate holds:
\begin{align*}
  \|T_{\Omega, i r}f\|_{L^{p,\infty}(w)}\lesssim [w]_{A_p}^{\min\{2, {p'}\}}\|\Omega\|_{L^\infty} \|f\|_{L^p(w)}.
\end{align*}
\item For $p=1$ and $w\in A_1$, we have
\begin{align*}
  \|T_{\Omega, i r}f\|_{L^{1,\infty}(w)}\lesssim [w]_{A_1} [w]_{A_\infty} \log_2([w]_{A_\infty}+1)\|\Omega\|_{L^\infty}\|f\|_{L^1(w)}.
\end{align*}
\item For all $1 \leq s < p$ and $w\in A_s$, we have
\begin{align*}
\| T_{\Omega, i r}f\|_{L^p(w)} \lesssim [w]_{A_s}\|f\|_{L^p(w)}
\begin{cases}
\|\Omega\|_{L^{q,1}\log L}, & 1<q<\infty, p >q',\\
\|\Omega\|_{L^\infty}, & 1<p<\infty.
\end{cases}
\end{align*}
\end{enumerate}
\end{theorem}

\bigskip
The paper is organized as follows.
Section \ref{L2} recalls the sparse domination principle and the definition of $A_p$ weights.
Section \ref{w(1,1)} is devoted to the proof of our sparse domination criterion, Theorem \ref{Thm}.
In Section \ref{sec-pfTb}, we provide the proofs of all results stated in the Applications subsection of the Introduction.

\bigskip
\noindent
{\bf Notation.}
Throughout this paper, $\mathscr{F}$ and $\mathscr{F}^{-1}$ (also denoted by $\widehat{f}$ for $\mathscr{F}f$) stand for the Fourier transform and its inverse, respectively.
For $p\in (1,\infty)$, we write $p' = p/(p-1)$ for its conjugate exponent, with the usual conventions $1'=\infty$ and $\infty'=1$.
We use the notation $X\lesssim Y$ to indicate that $X\le CY$ for some constant $C>0$ independent of the essential parameters on which $X$ and $Y$ depend.

\bigskip
\section{Preliminaries}\label{L2}
\setcounter{equation}{0}

In this section, we collect several definitions and lemmas that will be used in the proof of our main results.

\subsection{Sparse Domination Principle}

Let $1<r<\infty$.
Suppose that $\mathscr{T}$ is a bounded bilinear form on $L^r(\mathbb{R}^n) \times L^{r^{\prime}}(\mathbb{R}^n)$, associated with a kernel $\mathscr{K}(x,y)$ defined off the diagonal $\{(x,y)\in \mathbb{R}^n \times \mathbb{R}^n: x=y\}$.
More precisely, whenever $f_1 \in L^r(\mathbb{R}^n)$ and $f_2\in L^{r^{\prime}}(\mathbb{R}^n)$ have compact and disjoint supports,
$$
\mathscr{T}(f_1, f_2)=\int_{\mathbb{R}^n} \int_{\mathbb{R}^n}\mathscr{K}(x,y) f_1(y)  f_2(x)\,dy \, dx
$$
and the integral converges absolutely.

Assume that there exists $1<q \leq \infty$ such that the kernel $\mathscr{K}(x,y)$ admits a decomposition
\begin{equation*}
  \mathscr{K}(x,y)=\sum_{j\in\mathbb{Z}} \mathscr{K}_j(x,y)
\end{equation*}
where $\mathscr{K}_j$ satisfies
\begin{equation}\label{SS}\tag{SS}
\begin{aligned}
&\supp \mathscr{K}_j \subset\{(x,y) \in \mathbb{R}^n \times \mathbb{R}^n: 2^{j-2}<|x-y|<2^j \},\\
&[\mathscr{K}]_{0,q}:=\sup_{j\in\mathbb{Z}} 2^{\frac{j n}{q^{\prime}}} \sup_{x\in\mathbb{R}^n}(\|\mathscr{K}_j(x, x+\cdot)\|_q+\|\mathscr{K}_j(x+\cdot, x)\|_q)<\infty.
\end{aligned}
\end{equation}

Furthermore, we define the truncated forms associated with the above decomposition by
\begin{equation*}
  \mathscr{T}_\mu^\nu(h_1,h_2):=\int_{\mathbb{R}^n} \int_{\mathbb{R}^n} \sum_{\mu<j\leq \nu}\mathscr{K}_j(x,y)h_1(y)h_2(x)\,dy\,dx, \quad \mu,\nu\in\mathbb{Z} \cup\{-\infty,\infty\}.
\end{equation*}
We assume that these truncated forms satisfy the uniform boundedness condition
\begin{equation}\label{T}\tag{T}
  \quad C_{\mathscr{T}}(r)=:\sup_{\mu<\nu} (\|\mathscr{T}_\mu^\nu\|_{L^r(\mathbb{R}^n) \times L^{r^{\prime}}(\mathbb{R}^n) \rightarrow \mathbb{C}})<\infty.
\end{equation}

Throughout this article, by \emph{dyadic cubes} we mean the elements of a fixed dyadic lattice $\mathcal{D}$ in $\mathbb{R}^n$. The following definitions and notation concern local norms associated with stopping collections of dyadic cubes.

\begin{definition}[\cite{CCDO} Stopping collections]
Let $Q \in \mathcal{D}$ be a fixed dyadic cube in $\mathbb{R}^n$.
A collection $\mathscr{Q} \subset \mathcal{D}$ of dyadic cubes is called a \emph{stopping collection with top $Q$} if the following conditions are satisfied:

(1)The cubes in $\mathscr{Q}$ are pairwise disjoint and contained in $3Q$:
   \[
   L,L'\in \mathscr{Q},  L\cap L'\neq \emptyset \Rightarrow L=L',
   \qquad L\in \mathscr{Q} \Rightarrow L\subset 3Q;
   \]

(2)The cubes in $\mathscr{Q}$  enjoy the further separation properties:
$$
\begin{aligned}
& L,L^{\prime}\in\mathscr{Q}, |j_L-j_{L^{\prime}}|\geq 8 \Rightarrow 7L\cap 7L'=\emptyset, \\
& \bigcup_{L \in \mathscr{Q}: 3L\cap 2Q\neq\emptyset}9L\subset\bigcup_{L\in \mathscr{Q}} L=: \operatorname{sh}\mathscr{Q}.
\end{aligned}
$$
Here $j_L=\log_2 (\ell(L))$. The notation $\operatorname{sh}\mathscr{Q}$ for the union of the cubes in $\mathscr{Q}$ will also be used below.
\end{definition}

\begin{definition}[\cite{CCDO} Localized function spaces]\label{def:local-spaces}Let $1 \leq p \leq \infty$.
We define $\mathcal{Y}_p(\mathscr{Q})$ as the subspace of $L^p(\mathbb{R}^n)$ consisting of functions supported in $3Q$ with finite $\mathcal{Y}_p$-norm:
\[
\mathcal{Y}_p(\mathscr{Q}) := \{ h \in L^p(\mathbb{R}^n) : \supp h \subset 3Q, \|h\|_{\mathcal{Y}_p(\mathscr{Q})} < \infty \},
\]
where
\[
\|h\|_{\mathcal{Y}_p(\mathscr{Q})} :=
\begin{cases}
\displaystyle \max \{ \| h \chi_{\mathbb{R}^n \setminus \operatorname{sh}\mathscr{Q}} \|_\infty,
\sup_{L \in \mathscr{Q}} \inf_{x \in \widehat{L}} M_p h(x) \}, & p<\infty,\\[1mm]
\|h\|_\infty, & p=\infty.
\end{cases}
\]
Here $\widehat{L}$ denotes the (non-dyadic) $2^5$-fold dilate of $L$, and $M_p$ is the $p$-Hardy--Littlewood maximal function.
We next define $\mathcal{X}_p(\mathscr{Q}) \subset \mathcal{Y}_p(\mathscr{Q})$ as the subspace of functions that admit a decomposition into local blocks:
\[
\mathcal{X}_p(\mathscr{Q}) := \{ b \in \mathcal{Y}_p(\mathscr{Q}) : b = \sum_{L \in \mathscr{Q}} b_L, \ \supp b_L \subset L \}.
\]
Finally, the subspace of \emph{mean-zero local blocks} is defined by
\[
\dot{\mathcal{X}}_p(\mathscr{Q}) := \{ b \in \mathcal{X}_p(\mathscr{Q}) : \int_L b_L = 0, \ \forall L \in \mathscr{Q} \}.
\]
\end{definition}

For $b \in \mathcal{X}_p(\mathscr{Q})$ (resp. $b \in \dot{\mathcal{X}}_p(\mathscr{Q})$), we write $\|b\|_{\mathcal{X}_p(\mathscr{Q})}$ (resp. $\|\cdot\|_{\dot{\mathcal{X}}_p(\mathscr{Q})}$) to denote $\|b\|_{\mathcal{Y}_p(\mathscr{Q})}$. When the stopping collection $\mathscr{Q}$ is clear from the context, the subscript may be omitted. For the Calder\'{o}n-Zygmund decomposition associated to stopping collections, see \cite[Remark 2.2]{CCDO}.

Before stating the sparse domination principle, we introduce some notation for the localized truncated forms associated to the decomposition in \eqref{SS} as follows:
\begin{align*}
\mathscr{T}_{Q, \mu, \nu}(h_1, h_2) &:= \mathscr{T}_\mu^{\min \{j_Q,\nu\}}(h_1 \chi_Q, h_2)
= \mathscr{T}_\mu^{\min \{j_Q, \nu\}}(h_1 \chi_Q, h_2 \chi_{3 Q}),\\
\mathscr{T}_{\mathscr{Q},\mu,\nu}(h_1, h_2) &:= \mathscr{T}_{\mu}^{\nu}(h_1, h_2) - \sum_{\substack{L \in \mathscr{Q} \\ L \subset Q}} \mathscr{T}_{L, \mu,\nu}(h_1, h_2)=\mathscr{T}_{\mathscr{Q},\mu,\nu}(h_1\chi_Q, h_2\chi_{3Q}).
\end{align*}
The last equalities in the above definitions follow from the support property of $\mathscr{K}_j(x,y)$ in \eqref{SS}.

The following lemma provides a single-scale localized estimate.
\begin{lemma}\cite{CCDO}\label{local-e}
  Let $1<q\leq \infty$. Then, for all $k\geq1$, we have
  $$
  \sum_j\int|\mathscr{K}_j(x,y)||b_{j-k}(y)||h(x)|\,dy\,dx\lesssim [\mathscr{K}]_{0,q}|Q|\|b\|_{\mathcal{X}_1}\|h\|_{\mathcal{Y}_{q'}},
  $$
  where
$$
b_j=\sum_{L \in \mathscr{Q}, j_L=j} b_L.
$$
\end{lemma}

We now recall the abstract theorem of Conde-Alonso, Di Plinio et al., which is formulated for general non-convolution kernel operators.
\begin{lemma}\cite[Theorem C]{CCDO}\label{thm:C}
  Let $\mathscr{T}$ be a bilinear form satisfying \eqref{SS} and \eqref{T} above.
  Assume that there exist $1\leq p_1, p_2<\infty$ and a constant $C_L>0$ such that the estimates
  \begin{equation}\label{L}\tag{L}
    \begin{aligned}
      |\mathscr{T}_{\mathscr{Q}, \mu, \nu}(b, h)|&\leq C_L |Q|\|b\|_{\mathcal{\dot{X}}_{p_1}(\mathscr{Q})}\|h\|_{\mathcal{Y}_{p_2}(\mathscr{Q})},\\
      |\mathscr{T}_{\mathscr{Q}, \mu, \nu}(h, b)|&\leq C_L |Q|\|h\|_{\mathcal{Y}_{\infty}(\mathscr{Q})}\|b\|_{\mathcal{\dot{X}}_{p_2}(\mathscr{Q})},
    \end{aligned}
  \end{equation}
hold uniformly over all $\mu, \nu \in \mathbb{Z}$, all dyadic lattices $\mathcal{D}$, all $Q\in \mathcal{D}$ and all stopping collections $\mathscr{Q}\subset \mathcal{D}$ with top $Q$.
Then the estimate
\begin{equation*}
  \sup_{\mu, \nu \in \mathbb{Z}} |\mathscr{T}_{\mu}^{\nu}(f_1, f_2)|\lesssim [C_{\mathscr{T}}(r)+C_L]\sup_{\mathcal{S}}\sum_{Q\in\mathcal{S}}|Q|\Big(\frac{1}{|Q|}\int_Q|f_1|^{p_1}\Big)^{\frac{1}{p_1}}\Big(\frac{1}{|Q|}\int_Q|f_2|^{p_2}\Big)^{\frac{1}{p_2}},
\end{equation*}
holds for all $f_j\in L^{p_j}(\mathbb{R}^n)$ with compact support, $j=1,2$.
\end{lemma}

\subsection{Weights}\label{weights}

For $1<p<\infty$, the Muckenhoupt class $A_p$ consists of nonnegative, locally integrable functions $w$ satisfying
$$
[w]_{A_p}:=\sup_Q \Big(\frac{1}{|Q|}\int_Q w(x)dx\Big)\Big(\frac{1}{|Q|}\int_Q w(x)^{-\frac{1}{p-1}}dx\Big)^{p-1} < \infty,
$$
where the supremum is taken over all cubes $Q\subset \mathbb{R}^n$ with sides parallel to the coordinate axes.

For $p=1$, we say that $w \in A_1$ if there exists $\kappa>0$ such that
$$
Mw(x) \le \kappa w(x) \quad \text{a.e. } x \in \mathbb{R}^n,
$$
where $M$ denotes the Hardy-Littlewood maximal function.
The $A_1$ constant $[w]_{A_1}$ is defined as the infimum of all such $\kappa$ (see \cite{Gra}).

Finally, for $w \in A_\infty := \bigcup_{p \ge 1} A_p$, the $A_\infty$ constant, introduced by Fujii \cite{Fujii} and Wilson \cite{Wilson} is given by
$$
[w]_{A_\infty} := \sup_Q \frac{1}{w(Q)} \int_Q M(w\chi_Q)(x) dx < \infty,
$$
where $w(Q):=\int_Q w(x)dx$, and the supremum is taken over all cubes $Q$ as above.

\bigskip
\section{Proof of Theorem \ref{Thm}}\label{w(1,1)}
\setcounter{equation}{0}
In this section, we provide the proof of Theorem \ref{Thm}.
To this end, we construct a decomposition that meets the assumptions stated in Lemma \ref{thm:C}, namely \eqref{SS}, \eqref{T}, and \eqref{L}.

Let $1<q\leq \infty$ and suppose that $\Omega\in L^q(\mathbb{S}^{n-1})$.
Let $\varphi \in C_c^\infty(\mathbb{R}^n)$ be a radial function satisfying $0\leq \varphi\leq 1$, $\supp \varphi \subset \{2^{-2}\leq |x|\leq 1\}$ and $\sum_{j\in \mathbb{Z}} 2^{-jn}\varphi(2^{-j}x)=\frac{1}{|x|^n}$.
Define
\begin{equation}\label{Kj}
K_j(x,y)=K(x, y)2^{-jn}\varphi(2^{-j}(x-y))|x-y|^n.
\end{equation}
We then decompose the kernel as follows:
\begin{equation}
K_{\Omega}(x, y):=\Omega(x-y)K(x, y)=\sum_{j\in \mathbb{Z}} \Omega(x-y) K_j(x, y)=:\sum_{j\in \mathbb{Z}}  K_{\Omega,j}(x, y).
\end{equation}

We now verify the assumption \eqref{SS}.
Since $\varphi$ is supported in $\{2^{-2}\leq |x|\leq 1\}$, we have
$$
\supp K_{\Omega,j}(x, y) \subset\{(x,y) \in \mathbb{R}^n \times \mathbb{R}^n: 2^{j-2}<|x-y|<2^j  \}.
$$
In addition, by \eqref{Ksize}, for any fixed $x\in \mathbb{R}^n$, we have
\begin{align*}
\| K_{\Omega,j}(x, x+\cdot)\|_q &= \Big(\int_{\mathbb{R}^n} |\Omega(-z)K(x, x+z)2^{-jn}\varphi(2^{-j}(-z))|z|^n  |^q \,dz\Big)^{\frac{1}{q}}\\
&\lesssim \Big(\int_{\mathbb{S}^{n-1}}\int^{\infty}_{0} |\Omega(\theta)2^{-jn}\varphi(2^{-j}(r)) |^q \,  r^{n-1}dr d\theta\Big)^{\frac{1}{q}}\\
& \lesssim 2^{-\frac{jn}{q'}}\|\Omega\|_{L^q}.
\end{align*}
Similarly, $$\|  K_{\Omega,j}( x+\cdot, x)\|_q \lesssim  2^{-\frac{jn}{q'}}\|\Omega\|_{L^q}.$$
Hence, the desired estimate \eqref{SS} follows, with $[K]_{0,q}\lesssim \|\Omega\|_{L^q}$.

We now verify the assumption \eqref{T}.
Under the hypothesis of Theorem \ref{Thm}, it was shown in \cite[Corollary. 5.1]{DL} that
$$
\lambda |\{x\in R^n: |T_{\Omega}f(x)|>\lambda\}|\lesssim   \|\Omega\|_{L^q} \|f\|_1.
$$
Therefore, by the Marcinkiewicz interpolation theorem (see \cite[Theorem 1.3.2]{Gra} ) and duality, we obtain
$$\|T_{\Omega}f\|_r \lesssim \max\{\frac{1}{r-1},r\} \|\Omega\|_{L^q} \|f\|_r.$$
Define the truncated operators $T_{\Omega,\varepsilon}$ by
$$T_{\Omega, \varepsilon} f(x)= \int_{|x-y|>\varepsilon} K_{\Omega,j}(x,y) f(y)\,dy.$$
By \cite[Lemma 1]{LuZhang}, we have
$$\|T_{\Omega, \varepsilon}f\|_r \lesssim \max\{\frac{1}{r-1},r\} \|\Omega\|_{L^q} \|f\|_r.$$
Hence, the desired inequality \eqref{T} follows, with
$$C_{T_{\Omega}}(r)\lesssim \max\{\frac{1}{r-1},r\} \|\Omega\|_{L^q}.$$

Finally, we turn to the verification of the assumption \eqref{L}.
Let $\mathscr{Q}$ be a stopping collection with the top dyadic cube $Q$.
Define
$$
T_j (f)(x):= \int_{\mathbb{R}^n} \Omega(x-y) K_j(x, y) f(y)\,dy.
$$

Let $\{\varepsilon_j\} \in\{-1,0,1\}^{\mathbb{Z}}$ be a sequence of signs, and let $b \in \dot{\mathcal{X}}_1$. Denote
$$b_j=\sum_{L \in \mathscr{Q}, j_L=j} b_L.$$
We claim that the following estimate holds uniformly for all choices of $\{\varepsilon_j\} \in\{-1,0,1\}^{\mathbb{Z}}$:
\begin{equation}\label{forL}
\Big|\sum_{k \geq 1} \sum_{j\in \mathbb{Z}} \varepsilon_j \langle T_j (b_{j-k}), \bar{h}\rangle\Big| \lesssim   p' |Q|\|b\|_{\dot{\mathcal{X}}_1}\|h\|_{\mathcal{Y}_p}
\left\{\begin{array}{lr}
  \|\Omega\|_{L^{q,1}\log L}, &  1<q<\infty, p\geq q',\\
  \|\Omega\|_{L^\infty}, &  q=\infty, p>1.
\end{array}\right.
\end{equation}

The proof of \eqref{forL} will be given in Subsection \ref{sub1}.
We now use this estimate to verify the assumption \eqref{L}.
Denote by $\widetilde{K_{\Omega}}$ the kernel of adjoint operator of $T_{\Omega}$ .
Then
$$\widetilde{K_{\Omega}}(x,y)=\Omega(y-x)K(y,x).$$
Hence, following the argument on \cite[page 1271]{CCDO}, estimate \eqref{forL} implies the first condition in \eqref{L} with $p_1=1, p_2=p$, and the second condition in \eqref{L} with $p_2=p$.

With the assumptions \eqref{SS}, \eqref{T}, and \eqref{L} in hand, Lemma \ref{thm:C} together with \cite[Remark 2.4]{CCDO} yields Theorem \ref{Thm}.

\subsection{Proof of (\ref{forL})}\label{sub1}

In this subsection, we complete the proof of \eqref{forL}.
The argument is technically involved.
We begin by reducing the problem to the case $k \geq 10n$; this ensures that the regularity condition \eqref{Ksmooth} is satisfied (see \eqref{mudi} below).

By Lemma \ref{local-e} with $\mathscr{K}_j=K_{\Omega,j}$, we have, for each fixed $k\geq 1$,
\begin{equation}\label{allk}
\Big|\sum_{j\in \mathbb{Z}} \varepsilon_j \langle T_j (b_{j-k}), \bar{h}\rangle \Big| \lesssim  |Q|\|b\|_{\mathcal{X}_1}\|h\|_{\mathcal{Y}_{q'}}\|\Omega\|_{L^q}.
\end{equation}
Thus we have
\begin{equation}\label{10n}
\Big|\sum_{1\leq k<10 n} \sum_{j\in \mathbb{Z}} \varepsilon_j \langle T_j (b_{j-k}), \bar{h}\rangle \Big| \lesssim  |Q|\|b\|_{\mathcal{X}_1}\|h\|_{\mathcal{Y}_{q'}}\|\Omega\|_{L^q}.
\end{equation}

Define
$$
\mathcal{T}(b, h):=\sum_{k \geq 10 n} \sum_{j\in \mathbb{Z}} \varepsilon_j \langle T_j (b_{j-k}), \bar{h}\rangle.
$$
By \eqref{10n}, in order to prove \eqref{forL}, it suffices to establish the following estimate:
\begin{equation}\label{forL1}
|\mathcal{T}(b, h) | \lesssim   p' |Q|\|b\|_{\dot{\mathcal{X}}_1}\|h\|_{\mathcal{Y}_p}
\left\{\begin{array}{lr}
  \|\Omega\|_{L^{q,1}\log L}, &  1<q<\infty, p\geq q',\\
  \|\Omega\|_{L^\infty}, &  q=\infty, p>1.
\end{array}\right.
\end{equation}

To prove \eqref{forL1}, we smooth the kernel $K_j$ by convolution, obtaining a smooth approximation of the original operator, and the difference between the two can be controlled.

Let $\kappa_{\delta}(k)=[2\delta^{-1}\log_2 k]+10$, where $[a]$ denotes the integer part of $a$.
Let $\phi \in C^{\infty}_c(\mathbb{R}^n)$ be a nonnegative, radial function satisfying $\supp \phi \subset \{x\in \mathbb{R}^n: |x|\leq 1\}$ , and $\int_{\mathbb{R}^n} \phi(x)\,dx=1$.
Set $\phi_m(x)=2^{-mn} \phi(2^{-m}x)$.
Define
\begin{equation}\label{Kkj0}
K^{k}_j(x,y) = \int_{\mathbb{R}^n} \phi_{j-\kappa_{\delta}(k)}(x-z) K_j(z,y)\,dz.
\end{equation}

Since $K_j(z,y)$ is supported in $\{2^{j-2}\leq |z-y|\leq 2^{j}\}$ and $\phi_{j-\kappa_{\delta}(k)}(x)$ is supported in $\{|x|\leq 2^{j-\kappa_{\delta}(k)}\}$, it follows that $K^{k}_j(x,y)$ is supported in $\{2^{j-3}\leq |x-y|\leq 2^{j+1}\}$, and thus
\begin{equation}\label{Kkj}
  |K^{k}_j(x,y)|\lesssim 2^{-jn}\chi_{\{2^{j-3}\leq |x-y|\leq 2^{j+1}\}}.
\end{equation}

Define the operator $T^k_j$ by
$$
T^k_j (f)(x):= \int_{\mathbb{R}^n} \Omega(x-y) K^k_j(x, y) f(y)\,dy,
$$
and set
$$
\tilde{\mathcal{T}}(b, h):=\sum_{k \geq 10 n} \sum_j \varepsilon_j \langle T^k_j (b_{j-k}), \bar{h}\rangle.
$$

The following lemma shows that $\mathcal{T}(b, h)$ can be approximated by $\tilde{\mathcal{T}}(b, h)$ as desired.

\begin{lemma}\label{lemTerr}
Let $1<q\leq \infty$ and $\Omega\in L^q(\mathbb{S}^{n-1})$. With the above notations, we have
$$
|\tilde{\mathcal{T}}(b, h)-\mathcal{T}(b, h)|\lesssim  |Q|\|b\|_{\dot{\mathcal{X}}_1}\|h\|_{\mathcal{Y}_{q'}}\|\Omega\|_{L^q}.
$$
\end{lemma}

\begin{proof}
We first recall the definition of $b_j$.
It follows from the triangle inequality that
\begin{equation}\label{regular}
\begin{aligned}
  |\tilde{\mathcal{T}}(b, h)-\mathcal{T}(b, h)| = & \Big|\sum_{k \geq 10 n} \sum_j \varepsilon_j \langle (T^k_j-T_j) (b_{j-k}), \bar{h}\rangle \Big|\\
  \leq  & \sum_{k \geq 10 n} \sum_j \sum_{\substack{L \in \mathscr{Q} \\ j_L=j-k}}  \int_{\mathbb{R}^n} |(T^k_j-T_j) (b_{L})(x)| |h(x)|dx.
\end{aligned}
\end{equation}

By the definitions of $T^k_j$ and $T_j$, and the fact that $\int_{\mathbb{R}^n}\phi =1$, we obtain
\begin{align*}
  |(T^k_j-T_j)b_{L}(x)|=&\Big|\int_{\mathbb{R}^n} \Omega(x-y) \big(K^k_j(x, y)-K_j(x, y)\big) b_{L}(y)\,dy\Big|\\
  \leq &\int_{\mathbb{R}^n} |\Omega(x-y)| \int_{\mathbb{R}^n} \phi_{j-\kappa_{\delta}(k)}(z)| K_j(x-z,y)-K_j(x, y)|\,dz  |b_{L}(y)|\,dy.
\end{align*}

From the expression of $K_j(x,y)$, it follows that
\begin{align*}
  | K_j(x-z,y)-K_j(x, y)|&\leq  2^{-jn} |\varphi(2^{-j}(x-z-y))|x-z-y|^n-\varphi(2^{-j}(x-y))|x-y|^n| |K(x-z, y)|\\
  & \quad +  2^{-jn} |\varphi(2^{-j}(x-y))| |x-y|^n | K(x-z, y)-K(x,y)|.
\end{align*}

For the first term, note that $|z|\leq 2^{j-\kappa_{\delta}(k)}$ and, by the support condition of $\varphi_j$, we have $|x-y|\approx|x-z-y|$ and $2^{j-3}\leq|x-y|\leq 2^{j+1}$.
Then using \eqref{Ksize}, the first term is bounded by
\begin{align*}
  \frac{2^{-j}|z|}{|x-z-y|^n}\chi_{\{2^{j-3}\leq|x-y|\leq 2^{j+1}\}} \lesssim k^{-2}2^{-j n}\chi_{\{2^{j-3}\leq|x-y|\leq 2^{j+1}\}}.
\end{align*}

For the second term, since $|z|\leq 2^{j-\kappa_{\delta}(k)}$ and $2^{j-2}\leq|x-y|\leq 2^{j}$, we have
$2|z|<|x-y|$.
Then applying the regularity condition \eqref{Ksmooth}, the second term is bounded by
\begin{align*}
  \frac{|z|^{\delta}}{|x-y|^{n+\delta}}\chi_{\{2^{j-2}\leq|x-y|\leq 2^{j}\}} \lesssim k^{-2} 2^{-j n}\chi_{\{2^{j-3}\leq|x-y|\leq 2^{j+1}\}}.
\end{align*}

Combining the two estimates, it follows that
\begin{equation*}
  | K_j(x-z,y)-K_j(x, y)|\lesssim k^{-2} 2^{-j n}\chi_{\{2^{j-3}\leq|x-y|\leq 2^{j+1}\}}.
\end{equation*}
Therefore,
\begin{equation}\label{Tb}
\begin{aligned}
  \int_{\mathbb{R}^n} |(T^k_j-T_j) (b_{L})(x)| |h(x)|dx
  \lesssim  k^{-2} 2^{-j n}   \int_{\mathbb{R}^n}  \int_{2^{j-3}\leq|x-y|\leq 2^{j+1}} |\Omega(x-y)|   |b_{L}(y)|\,dy |h(x)|dx.
\end{aligned}
\end{equation}

We claim that for each $L\in \mathscr{Q}$ with $j_L=j-k$, one has
\begin{equation}\label{O1}
   \int_{\mathbb{R}^n}  \int_{2^{j-3}\leq|x-y|\leq 2^{j+1}} |\Omega(x-y)|   |b_{L}(y)|\,dy |h(x)|dx
  \lesssim  2^{jn} \|b_{L}\|_{1}\|h\|_{\mathcal{Y}_{q'}}\|\Omega\|_{L^q}.
\end{equation}

To verify this, it suffices to show that for all $j\geq j_L$
$$
\sup_{y\in L} \int_{2^{j-3}\leq|x-y|\leq 2^{j+1}} |\Omega(x-y)| |h(x)|dx
  \lesssim  2^{jn} \|h\|_{\mathcal{Y}_{q'}}\|\Omega\|_{L^q}.
$$
Indeed, by H\"{o}lder's inequality and the support condition, we have
\begin{equation}\label{qq'}
\begin{aligned}
  \int_{2^{j-3}\leq|x-y|\leq 2^{j+1}} |\Omega(x-y)| |h(x)|dx
  \leq & \Big(\int_{2^{j-3}\leq|x|\leq 2^{j+1}} |\Omega(x)|^q dx \Big)^{\frac{1}{q}} \Big(\int_{B(y,2^{j+10})} |h(x)|dx \Big)^{\frac{1}{q'}}\\
  \lesssim & 2^{jn} \, (\inf_{\hat{L}} M_{q'}h)\,\|\Omega\|_{L^q}\leq 2^{jn} \|h\|_{\mathcal{Y}_{q'}}\|\Omega\|_{L^q}.
\end{aligned}
\end{equation}
Here we used that $B(y,2^{j+10})$ contains the dilated cube $\hat{L}$.
Hence \eqref{O1} follows.

Substituting \eqref{O1} into \eqref{Tb}, and then into \eqref{regular}, and using the fact that $\|b_{L}\|_1\lesssim |L|\|b\|_{\dot{\mathcal{X}}_1}$ for $L\in \mathscr{Q}$, and that the cubes \(L\) are disjoint, we obtain
\begin{align*}
  |\tilde{\mathcal{T}}(b, h)-\mathcal{T}(b, h)|
  \lesssim & \sum_{k \geq 10 n} \sum_j \sum_{L\in \mathscr{Q}, j_L=j-k} k^{-2}  \|b_{L}\|_{1} \|h\|_{\mathcal{Y}_{q'}}\|\Omega\|_{L^q} \\
  \lesssim & \sum_{k \geq 10 n} \sum_j \sum_{L\in \mathscr{Q}, j_L=j-k} k^{-2}   |L| \|b\|_{\dot{\mathcal{X}}_1}  \|h\|_{\mathcal{Y}_{q'}}\|\Omega\|_{L^q} \\
  \lesssim &  |Q|\|b\|_{\dot{\mathcal{X}}_{1}} \|h\|_{\mathcal{Y}_{q'}}\|\Omega\|_{L^q}.
\end{align*}
This completes the proof of Lemma \ref{lemTerr}.
\end{proof}

By Lemma \ref{lemTerr} and the fact that $p\geq 1$, to verify \eqref{forL1}, it suffices to show
\begin{equation}\label{forL2}
|\tilde{\mathcal{T}} (b,h)| \lesssim  p' |Q|\|b\|_{\dot{\mathcal{X}}_1}\|h\|_{\mathcal{Y}_p}
\left\{\begin{array}{lr}
  \|\Omega\|_{L^{q,1}\log L}, &  1<q<\infty, p\geq q',\\
  \|\Omega\|_{L^\infty}, &  q=\infty, p>1.
\end{array}\right.
\end{equation}

To this end, we decompose the operators $T^k_j$ into a regularized part and a remainder.
Let $P_m$ denote the convolution operator with kernel $\phi_m$.
Then we write
$$T^k_j=P_{j-k+\kappa_{\delta}(k)}T^k_j+(I-P_{j-k+\kappa_{\delta}(k)})T^k_j.$$

The following lemma controls the contribution of the regularized part.

\begin{lemma}\label{Pbh}
Let $1<q\leq \infty$ and $\Omega\in L^q(\mathbb{S}^{n-1})$. With the above notations, we have
$$
\tilde{\mathcal{P}}(b, h):=\Big|\sum_{k \geq 10 n} \sum_j \varepsilon_j \langle P_{j-k+\kappa_{\delta}(k)}T^k_j (b_{j-k}), \bar{h}\rangle \Big|\lesssim |Q|  \|b\|_{\mathcal{\dot{X}}_{1}}\|h\|_{\mathcal{Y}_{q'}}\|\Omega\|_{L^q}.
$$
\end{lemma}

\begin{proof}
  Let $L\in \mathscr{Q}$ with $j_L=j-k$. Applying Fubini's theorem, we can write
\begin{align*}
  &P_{j-k+\kappa_{\delta}(k)} T^{k}_{j}b_L(x)\\
  &=\int_{\mathbb{R}^n} \phi_{j-k+\kappa_{\delta}(k)}(x-w) T^{k}_{j}b_L(w)\,dw\\
  &=\int_{\mathbb{R}^n}  \phi_{j-k+\kappa_{\delta}(k)}(x-w) \int_{\mathbb{R}^n} \Omega(w-y) K^k_j(w,y) b_L(y)\,dy \,dw\\
  &=\int_{\mathbb{R}^n} \int_{\mathbb{R}^n}  \phi_{j-k+\kappa_{\delta}(k)}(x-w)  \Omega(w-y) K^k_j(w,y) \,dw b_L(y)\,dy\\
  &= \int_{\mathbb{R}^n} A_{j,k}(x,y)b_L(y)\,dy,
\end{align*}
where $$A_{j,k}(x,y)=\int_{\mathbb{R}^n}  \phi_{j-k+\kappa_{\delta}(k)}(x-w)  \Omega(w-y) K^k_j(w,y) \,dw.$$

In view of the cancellation of $b_L$, and applying Fubini's theorem, we can write
\begin{align*}
  \int_{\mathbb{R}^n}P_{j-k+\kappa_{\delta}(k)} T^{k}_{j} b_L(x)h(x)dx
  = \int_{\mathbb{R}^{n}} \int_{\mathbb{R}^{n}} \Big(A_{j,k}(x, y)-A_{j,k}(x, y_L)\Big)h(x)dx  b_L(y)dy,
\end{align*}
where \( y_L \) denotes the center of \( L \).

Recalling the supports of $\phi_{j-k+\kappa_{\delta}(k)}$ and $K^k_j$, we have
$$
\supp A_{j,k}(x,y) \subset \{z: 2^{j-4}\leq |z|\leq2^{j+2}\}.
$$
Since $y\in L$, we have $x\in B(y,2^{j+10})$.
Applying H\"{o}lder's inequality, we obtain
\begin{equation}\label{Aq}
\begin{aligned}
  &\Big|\int_{\mathbb{R}^{n}} \int_{\mathbb{R}^{n}} \Big(A_{j,k}(x, y)-A_{j,k}(x, y_L)\Big)h(x)dx  b_L(y)dy\Big|\\
  &\leq  \sup_{y\in L}\Big |\int_{\mathbb{R}^{n}} \Big(A_{j,k}(x, y)-A_{j,k}(x, y_L)\Big)h(x)dx  \Big| \|b_L\|_1\\
  &\leq  \sup_{y\in L}\|A_{j,k}(x, y)-A_{j,k}(x, y_L)\|_q \Big(\int_{B(y,2^{j+10})} |h(x)|dx \Big)^{\frac{1}{q'}}  \|b_L\|_1\\
  &\leq  \sup_{y\in L}\|A_{j,k}(x, y)-A_{j,k}(x, y_L)\|_q  2^{\frac{jn}{q'}} \|h\|_{\mathcal{Y}_{q'}}  \|b_L\|_1,
\end{aligned}
\end{equation}
where the last inequality uses the same argument as in \eqref{qq'}.

We now turn to estimating $\|A_{j,k}(x, y)-A_{j,k}(x, y_L)\|_q$.
By the change of variables $w=x+y-z$, we can rewrite $A_{j,k}(x, y)$ as
$$A_{j,k}(x,y)=\int_{\mathbb{R}^n}  \phi_{j-k+\kappa_{\delta}(k)}(z-y)  \Omega(x-z) K^k_j(x+y-z,y) \,dz.$$
Therefore, we have
\begin{align*}
&A_{j,k}(x,y)-A_{j,k}(x,y_L)\\
&=\int_{\mathbb{R}^n} \Omega(x-z) \Big(\phi_{j-k+\kappa_{\delta}(k)}(z-y)   K^k_j(x+y-z,y)-\phi_{j-k+\kappa_{\delta}(k)}(z-y_L)   K^k_j(x+y_L-z,y_L)\Big) \,dz\\
&=\int_{\mathbb{R}^n} \Omega(x-z) \Big(\phi_{j-k+\kappa_{\delta}(k)}(z-y)   -\phi_{j-k+\kappa_{\delta}(k)}(z-y_L)\Big)K^k_j(x+y-z,y) \,dz\\
&\quad +\int_{\mathbb{R}^n} \Omega(x-z)  \phi_{j-k+\kappa_{\delta}(k)}(z-y_L) \Big(K^k_j(x+y-z,y)-  K^k_j(x+y_L-z,y_L)\Big) \,dz\\
&=:I_1+I_2.
\end{align*}

For the term $I_1$, by \eqref{Kkj}, we have
\begin{align*}
|I_1|\lesssim & 2^{-j n} \int_{2^{j-3}\leq |x-z|\leq 2^{j+1}} |\Omega(x-z)| |\phi_{j-k+\kappa_{\delta}(k)}(z-y)   -\phi_{j-k+\kappa_{\delta}(k)}(z-y_L)| \,dz\\
\lesssim & 2^{-j n} \int_{2^{j-3}\leq |x-z|\leq 2^{j+1}} |\Omega(x-z)| \big|\int_0^1 \langle y-y_L, \nabla (\phi_{j-k+\kappa_{\delta}(k)}) (z-ty-(1-t)y_L)dt\big| \,dz.
\end{align*}
Thus by Minkowski's inequality, we have
\begin{align*}
  \|I_1\|_q \lesssim & 2^{-j n} \int_{\mathbb{R}^n} (\int_{2^{j-3}\leq |x-z|\leq 2^{j+1}} |\Omega(x-z)|^q dx)^{\frac{1}{q}} \Big|\int_0^1 \langle y-y_L, \nabla (\phi_{j-k+\kappa_{\delta}(k)}) (z-ty-(1-t)y_L)dt\Big| \,dz\\
  \lesssim & 2^{-\frac{j n}{q'}} \|\Omega\|_{L^q} \int_{\mathbb{R}^n} \Big|\int_0^1 \langle y-y_L, \nabla (\phi_{j-k+\kappa_{\delta}(k)}) (z-ty-(1-t)y_L)dt\Big| \,dz\\
    \lesssim & 2^{j-k} 2^{-\frac{j n}{q'}} \|\Omega\|_{L^q}  \| \nabla (\phi_{j-k+\kappa_{\delta}(k)})\|_1\\
        \lesssim & 2^{-\kappa_{\delta}(k)} 2^{-\frac{j n}{q'}} \|\Omega\|_{L^q}.
\end{align*}

For the term $I_2$, write $K^k_j(x+y-z,y)-  K^k_j(x+y_L-z,y_L)$ as
$$K^k_j(x+y-z,y)- K^k_j(x+y_L-z,y)+  K^k_j(x+y_L-z,y) - K^k_j(x+y_L-z,y_L).$$
Notice that $|K_j(x,y)|\lesssim 2^{-j n}\chi_{\{2^{j-2}\leq |x-y|\leq 2^j\}}(x-y)$.
Therefore, by the definition of $K^k_j$ in \eqref{Kkj0}, we have
\begin{align*}
  &|K^k_j(x+y-z,y)- K^k_j(x+y_L-z,y)|\\
  & =\Big|    \int_{\mathbb{R}^n} \Big(\phi_{j-\kappa_{\delta}(k)}(x+y-z-w)-\phi_{j-\kappa_{\delta}(k)}(x+y_L-z-w)\Big) K_j(w,y) dw     \Big|\\
  &\lesssim 2^{-j n} \int_{\mathbb{R}^n} \int_0^1 |\langle y-y_L, \nabla(\phi_{j-\kappa_{\delta}(k)})(x+ty+(1-t)y_L-z-w) \rangle| dtdw\\
  &\lesssim  2^{j-k} 2^{-j n}   \|\nabla(\phi_{j-\kappa_{\delta}(k)})\|_1 \\
  &\lesssim  2^{\kappa_{\delta}(k)-k} 2^{-j n}.
\end{align*}

For $K^k_j(x+y_L-z,y) - K^k_j(x+y_L-z,y_L)$, we have
\begin{align*}
  &|K^k_j(x+y_L-z,y) - K^k_j(x+y_L-z,y_L)|\\
  &=\Big|    \int_{\mathbb{R}^n} \phi_{j-\kappa_{\delta}(k)}(x+y_L-z-w)\Big(K_j(w,y)-K_j(w,y_L)\Big)  dw     \Big|.
\end{align*}

By the definition of $K_j$ in \eqref{Kj}, we have
\begin{align*}
  |K_j(w,y)-K_j(w,y_L)|&= |K(w, y)2^{-jn}\varphi(2^{-j}(w-y))|w-y|^n-K(w, y_L)2^{-jn}\varphi(2^{-j}(w-y_L))|w-y_L|^n|\\
  &\leq 2^{-jn}\varphi(2^{-j}(w-y))|w-y|^n|K(w, y)-K(w, y_L)|\\
  &\quad +|2^{-jn}\varphi(2^{-j}(w-y))|w-y|^n-2^{-jn}\varphi(2^{-j}(w-y_L))|w-y_L|^n| |K(w, y_L)|.
\end{align*}

For the first term, since $|y-y_L|\leq 2^{j-k}$ and $k\geq 10 n$, by the support of $\varphi$, we have $2|y-y_L|<|w-y|$.
Applying the regularity condition \eqref{Ksmooth}, the first term is bounded by
\begin{equation}\label{mudi}
  2^{-jn}\varphi(2^{-j}(w-y))|w-y|^n|K(w, y)-K(w, y_L)|
  \lesssim  \frac{|y-y_L|^\delta}{|w-y|^{n+\delta}}
  \lesssim  2^{-jn-k\delta}.
\end{equation}

For the second term, since $|y-y_L|\leq 2^{j-k}$ and $k\geq 10 n$, by the support of $\varphi$, we have $|w-y|\approx|w-y_L|$.
Applying the size condition \eqref{Ksize}, the second term is bounded by
\begin{align*}
 \frac{1}{|w-y_L|^{n}} \int_0^1|\langle y-y_L, \nabla (2^{-jn}\varphi(2^{-j}(\cdot))|\cdot|^n)(w-ty-(1-t)y_L) | dt
  \lesssim 2^{-jn-k}.
\end{align*}

Combining the two bounds, we obtain
\begin{align*}
  |K^k_j(x+y_L-z,y) - K^k_j(x+y_L-z,y_L)|\lesssim 2^{-jn -k\delta}.
\end{align*}

Observe that the integrand in $I_2$ is supported in $2^{j-3}\leq |x-z|\leq 2^{j+1}$. Applying Minkowski's inequality, we get
\begin{align*}
  \|I_2\|_q\lesssim & (2^{ -k\delta} +2^{\kappa_{\delta}(k)-k} )2^{-j n}\int_{\mathbb{R}^n}(\int_{2^{j-3}\leq |x-z|\leq 2^{j+1}} |\Omega(x-z)|^q)^{\frac{1}{q}}  \phi_{j-k+\kappa_{\delta}(k)}(z-y_L) \,dz\\
  \lesssim & (2^{-k\delta} +2^{\kappa_{\delta}(k)-k}) 2^{-\frac{jn}{q'}} \|\Omega\|_{L^q}.
\end{align*}

Consequently, by the estimates of $I_1$ and $I_2$, we have
\begin{align*}
\|A_{j,k}(x, y)-A_{j,k}(x, y_L)\|_q
\lesssim  (2^{-\kappa_{\delta}(k)}+2^{-k\delta} +2^{\kappa_{\delta}(k)-k})2^{-\frac{jn}{q'}}\|\Omega\|_{L^q}.
\end{align*}

By \eqref{Aq}, we have
$$
\Big| \int_{\mathbb{R}^n}P_{j-k+\kappa_{\delta}(k)} T^{k}_{j} b_L(x)h(x)dx \Big|\lesssim  (2^{-\kappa_{\delta}(k)}+2^{-k\delta} +2^{\kappa_{\delta}(k)-k})\|b_L\|_1\|h\|_{\mathcal{Y}_{q'}}  \|\Omega\|_{L^q}.
$$

Applying the triangle inequality and using that $\|b_{L}\|_1\lesssim |L|\|b\|_{\dot{\mathcal{X}}_1}$ for $L\in \mathscr{Q}$, together with the disjointness of the cubes $L$, we obtain
\begin{align*}
  |\tilde{\mathcal{P}}(b, h)|\leq & \sum_{k \geq 10 n} \sum_j \sum_{L\in \mathscr{Q}, j_L=j-k} | \langle P_{j-k+\kappa_{\delta}(k)}T^k_j (b_{L}), \bar{h}\rangle|\\
  \lesssim & \sum_{k \geq 10 n} \sum_j \sum_{L\in \mathscr{Q}, j_L=j-k}(2^{-\kappa_{\delta}(k)}+2^{-k\delta} +2^{\kappa_{\delta}(k)-k})\|b_L\|_1\|h\|_{\mathcal{Y}_{q'}}  \|\Omega\|_{L^q}\\
    \lesssim & \sum_{k \geq 10 n} \sum_j \sum_{L\in \mathscr{Q}, j_L=j-k}(2^{-\kappa_{\delta}(k)}+2^{-k\delta} +2^{\kappa_{\delta}(k)-k})|L|\|b\|_{\dot{\mathcal{X}}_1}\|h\|_{\mathcal{Y}_{q'}}  \|\Omega\|_{L^q}\\
    \lesssim   &  |Q|\|b\|_{\dot{\mathcal{X}}_1}\|h\|_{\mathcal{Y}_{q'}}  \|\Omega\|_{L^q}.
\end{align*}
This completes the proof of Lemma \ref{Pbh}.
\end{proof}

Denote $$\tilde{\mathcal{B}} (b,h):=  \Big|\sum_{k \geq 10 n} \sum_j \varepsilon_j \langle (I-P_{j-k+\kappa_{\delta}(k)})T^k_j (b_{j-k}), \bar{h}\rangle \Big|.$$
By Lemma \ref{Pbh}, to show \eqref{forL2}, it suffices to prove
\begin{equation}\label{forL3}
\begin{aligned}
\tilde{\mathcal{B}} (b,h)\lesssim      p' |Q|\|b\|_{\dot{\mathcal{X}}_1}\|h\|_{\mathcal{Y}_p}
\left\{\begin{array}{lr}
  \|\Omega\|_{L^{q,1}\log L}, &  1<q<\infty, p\geq q',\\
  \|\Omega\|_{L^\infty}, &  q=\infty, p>1.
\end{array}\right.
\end{aligned}
\end{equation}

We next decompose the kernel $\Omega$ into bounded and unbounded parts.
For each fixed integer $k \geq 10 n$, define
$$
D_{k}=\{\theta\in \mathbb{S}^{n-1}: |\Omega(\theta)|>2^{\gamma k}\|\Omega\|_{L^q}\}, \quad \Omega_k=\Omega \chi_{\mathbb{S}^{n-1}\backslash D_{k}},\quad \mho_k=\Omega \chi_{ D_{k}},
$$
where $\gamma>0$ will be fixed later.
Correspondingly, we write
$$
T^k_j (f)(x)= U^{k}_j (f)(x) +V^{k}_j (f)(x)
$$
with
$$
U^{k}_j (f)(x)=\int_{\mathbb{R}^n} \Omega_k(x-y) K^k_j(x, y) f(y)\,dy;\,
V^{k}_j (f)(x)=\int_{\mathbb{R}^n} \mho_k(x-y) K^k_j(x, y) f(y)\,dy.
$$

The following lemma gives basic estimates involving $U^{k}_j$ and $V^{k}_j$.

\begin{lemma}\label{UVbh}
Let $1<q\leq \infty$ and $\Omega\in L^q(\mathbb{S}^{n-1})$. With the above notations, we have
\begin{equation}\label{Ukbh}
\mathcal{U}^{k}(b,h):=  |\sum_j \varepsilon_j \langle (I-P_{j-k+\kappa_{\delta}(k)})U^{k}_j (b_{j-k}), \bar{h}\rangle |\lesssim  |Q| \|b\|_{\dot{\mathcal{X}}_{1}} \|h\|_{\mathcal{Y}_{1}}\|\Omega_k\|_{L^{\infty}},
\end{equation}
and
\begin{equation}\label{Vkbh}
\mathcal{V}^{k}(b,h):=  \big|\sum_j \varepsilon_j \langle (I-P_{j-k+\kappa_{\delta}(k)})V^{k}_j (b_{j-k}), \bar{h}\rangle \big|\lesssim  |Q| \|b\|_{\dot{\mathcal{X}}_{1}} \|h\|_{\mathcal{Y}_{q'}}\|\mho_k\|_{L^{q}}.
\end{equation}
\end{lemma}

We next provide a decay estimate for the bounded part.

\begin{lemma}\label{Ubh}
  Let $1<q\leq \infty$ and $\Omega\in L^q(\mathbb{S}^{n-1})$. With the above notations, we have
  \begin{equation}\label{Ukbh2}
  \mathcal{U}^{k}(b,h)\lesssim k^{200n\delta^{-1}}  2^{-k/8}   |Q| \|b\|_{\dot{\mathcal{X}}_{1}} \|h\|_{\mathcal{Y}_{\infty}}\|\Omega_k\|_{L^{\infty}}.
  \end{equation}
\end{lemma}

The proofs of Lemmas \ref{UVbh} and \ref{Ubh} will be given in Subsection \ref{UV}.
Now, by Lemmas \ref{UVbh} and \ref{Ubh}, we complete the proof of \eqref{forL3}.

Interpolating between \eqref{Ukbh} and \eqref{Ukbh2}
with respect to $h$ via the Riesz-Thorin theorem
(see \cite[Theorem 1.3.4]{Gra}), we obtain
\begin{align*}
  \mathcal{U}^k(b,h) \lesssim (k^{200n\delta^{-1}}  2^{-k/8})^{\frac{1}{p'}}   |Q| \|b\|_{\dot{\mathcal{X}}_{1}} \|h\|_{\mathcal{Y}_{p}}\|\Omega_k\|_{L^{\infty}}.
\end{align*}
Taking $\gamma=1/(16p')$, it follows that for any $p>1$,
\begin{equation}\label{eukbh}
\begin{aligned}
  \sum_{k\geq 10n} \mathcal{U}^k(b,h)\lesssim &  \sum_{k\geq 10n}(k^{200n\delta^{-1}}  2^{-k/8})^{\frac{1}{p'}}  |Q| \|b\|_{\dot{\mathcal{X}}_{1}} \|h\|_{\mathcal{Y}_{p}}\|\Omega_k\|_{L^{\infty}}\\
  \lesssim &  \sum_{k\geq 10n} (k^{200n\delta^{-1}}  2^{-k/16})^{\frac{1}{p'}}   |Q| \|b\|_{\dot{\mathcal{X}}_{1}} \|h\|_{\mathcal{Y}_{p}}\|\Omega\|_{L^q}\\
  \lesssim &  p' |Q| \|b\|_{\dot{\mathcal{X}}_{1}} \|h\|_{\mathcal{Y}_{p}}\|\Omega\|_{L^q}.
\end{aligned}
\end{equation}

For $1<q<\infty$, as shown in \cite[page 1273]{CCDO},
$$\sum_{k\geq 10n} \|\mho_k\|_{L^{q}}\lesssim   p'\|\Omega\|_{L^{q,1}\log L}.$$
It then follows from \eqref{Vkbh} that
\begin{equation}\label{evkbh}
  \sum_{k\geq 10n} \mathcal{V}^k(b,h) \lesssim   \sum_{k\geq 10n} \|\mho_k\|_{L^{q}} |Q| \|b\|_{\dot{\mathcal{X}}_{1}} \|h\|_{\mathcal{Y}_{q'}}
  \lesssim   p' |Q| \|b\|_{\dot{\mathcal{X}}_{1}} \|h\|_{\mathcal{Y}_{q'}}\|\Omega\|_{L^{q,1}\log L}.
\end{equation}
Moreover, when $q=\infty$, since $\mho_k=0$, we have $\mathcal{V}^k(b,h)\equiv 0$.

Combining \eqref{eukbh} with \eqref{evkbh}, we obtain
\begin{align*}
  \tilde{\mathcal{B}}(b,h) &\leq \sum_{k\geq 10n} \mathcal{U}^k(b,h)+\sum_{k\geq 10n} \mathcal{V}^k(b,h)\\
  &\lesssim    p' |Q|\|b\|_{\dot{\mathcal{X}}_1}\|h\|_{\mathcal{Y}_p}
\left\{\begin{array}{lr}
  \|\Omega\|_{L^{q,1}\log L}, &  1<q<\infty, p\geq q',\\
  \|\Omega\|_{L^\infty}, &  q=\infty, p>1.
\end{array}\right.
\end{align*}
This proves \eqref{forL3}, and hence completes the proof of \eqref{forL}.

\subsection{Proofs of Lemmas \ref{UVbh} and \ref{Ubh}}\label{UV}
\begin{proof}[Proof of Lemma \ref{UVbh}]
We first prove \eqref{Vkbh}. By Fubini's theorem, the kernel of $(I-P_{j-k+\kappa_{\delta}(k)})V^{k}_j $ can be written as
$$\tilde{A}_{j,k}(x,y)=\mho_k(x-y) K^k_j(x, y)-\int_{\mathbb{R}^n}  \phi_{j-k+\kappa_{\delta}(k)}(x-w)  \mho_k(w-y) K^k_j(w,y) \,dw.$$

For each fixed $x\in \mathbb{R}^n$, applying \eqref{Kkj}, we obtain
\begin{align*}
  \|\mho_k(-\cdot) K^k_j(x, x+\cdot)\|_{q} = &\Big(\int_{\mathbb{R}^n} |\mho_k(-z) K^k_j(x, x+z)|^q\Big)^{\frac{1}{q}}\\
  \lesssim  &2^{-jn} \Big(\int_{2^{j-3}\leq |z|\leq 2^{j+1}} |\mho_k(-z) |^q\Big)^{\frac{1}{q}}\\
  \lesssim  &2^{-\frac{jn}{q'}}  \|\mho_k\|_{L^{q}}.
\end{align*}
Moreover, using \eqref{Kkj}, we have
\begin{align*}
 \int_{\mathbb{R}^n}  \phi_{j-k+\kappa_{\delta}(k)}(x-w)  \mho_k(w-y) K^k_j(w,y) \,dw
  \lesssim 2^{-jn}  \phi_{j-k+\kappa_{\delta}(k)}\ast \mho_k(\cdot-y)\chi_{\{2^{j-3}\leq |\cdot-y|\leq 2^{j+1}\}}(x).
\end{align*}
Then, using Minkowski's inequality and the fact that $\int_{\mathbb{R}^n} \phi =1$, we obtain
\begin{align*}
  \big\|\int_{\mathbb{R}^n}  \phi_{j-k+\kappa_{\delta}(k)}(x-w)  \mho_k(w-x-\cdot) K^k_j(w, x+\cdot) \,dw \big\|_q
    \lesssim  &2^{-jn} \Big(\int_{2^{j-3}\leq |z|\leq 2^{j+1}} |\mho_k(-z) |^q\Big)^{\frac{1}{q}}\\
  \lesssim  &2^{-\frac{jn}{q'}}  \|\mho_k\|_{L^{q}}.
\end{align*}
It follows that
$$\|\tilde{A}_{j,k}(x,x+\cdot)\|_q \lesssim 2^{-\frac{jn}{q'}} \|\mho_k\|_{L^{q}}, $$
and similarly, $\|\tilde{A}_{j,k}(x+\cdot, x)\|_q \lesssim 2^{-\frac{jn}{q'}} \|\mho_k\|_{L^{q}}.$
Hence,
$$[\tilde{A}]_{0,q}:=\sup_{j}2^{\frac{jn}{q'}} \sup_{x\in \mathbb{R}^n} (\|\tilde{A}_{j,k}(x,x+\cdot)\|_q+\|\tilde{A}_{j,k}(x+\cdot,x)\|_q)\lesssim    \|\mho_k\|_{L^{q}}.$$
Applying Lemma \ref{local-e} with $\mathscr{K}_j=\tilde{A}_{j,k}$, we obtain \eqref{Vkbh}.

Finally, since $\Omega_k\in L^{\infty}(\mathbb{S}^{n-1})$, the same argument yields \eqref{Ukbh}. This completes the proof of Lemma \ref{UVbh}.
\end{proof}

\begin{proof}[Proof of Lemma \ref{Ubh}]
To complete the proof, we perform a microlocal decomposition of the kernel $\Omega_k$
as in \cite[p.~1654]{DL} (see also \cite{Seeger, Seeger2} for a related decomposition).
For this purpose, we construct a partition of unity on the support of $\Omega_k$.
Fixed $k \geq 10 n$. Let $\Theta_k = \{ e^k_l \}_l$ be a collection of unit vectors on $\supp \Omega_k$ satisfying the following conditions:
\begin{itemize}
    \item[(i)] $|e^k_l - e^k_{\tilde{l}}| \geq 2^{-k\alpha - 10 n}$  for $l \neq \tilde{l}$;
    \item[(ii)] For any $x' \in \supp \Omega_k$, there exists an $e^k_l$ such that $|e^k_l - x'| \leq 2^{-k\alpha - 10 n}$.
\end{itemize}
Here, the constant $0 < \alpha < 1$ will be chosen later.

To construct such a set $\Theta_k$, we may take a maximal collection $\{ e^k_l \}_l$ satisfying condition (i).
It follows that card$(\Theta_k) \lesssim 2^{k\alpha(n-1)}$.
Moreover, for each $x' \in \supp \Omega_k$, there are only finitely many $e^k_l$ satisfying $|e^k_l - x'| \leq 2^{-k\alpha - 10 n}$.

We now construct an associated partition of unity on $\supp \Omega_k$. Let $\zeta$ be a smooth, nonnegative, radial function satisfying $\zeta(r) = 1$ for $|r| \leq \frac{1}{2}$ and $\zeta(r) = 0$ for $|r| > 1$.
Define
\[
\tilde{\Gamma}^k_l (x) = \zeta \big(2^{k\alpha} ( x' - e^k_l) \big),
\]
and set
\[
\Gamma^k_l(x) = \tilde{\Gamma}^k_l(x)  \big(\sum_l \tilde{\Gamma}^k_l(x)\big)^{-1} .
\]
Clearly, $\Gamma^k_l$ is homogeneous of degree $0$ and satisfies
\[
\sum_l \Gamma^k_l(x) = 1, \quad \text{for all } x \neq 0 \text{ and all } k \geq 10 n.
\]

Next, we define the operator $U^{k,l}_{j}$ by
\begin{equation}\label{Uklj}
U^{k,l}_{j} f(x) = \int_{\mathbb{R}^n} \Gamma^k_l(x-y) \cdot \Omega_k(x-y) K^k_j(x,y) \cdot f(y) \, dy.
\end{equation}
Therefore, we have
$
U^{k}_{j} = \sum_l U^{k,l}_{j}.
$

In the following, we need to separate the phase into different directions.
To this end, we define the multiplier operator
\[
\widehat{\Psi_{k,l} f}(\xi) = \Phi\big(2^{k\alpha} \langle e^k_l, \xi' \rangle\big) \hat{f}(\xi),
\]
where $\Phi$ is a smooth, nonnegative, radial function satisfying
\[
0 \leq \Phi(x) \leq 1, \quad \Phi(x) = 1 \text{ for } |x| \leq 2, \quad \Phi(x) = 0 \text{ for } |x| > 4.
\]

We split $U^{k,l}_{j}$ into two parts as follows:
\[
U^{k,l}_{j} = \Psi_{k,l} U^{k,l}_{j} + (I - \Psi_{k,l}) U^{k,l}_{j}.
\]
Define
\begin{align*}
   \mathcal{U}_1^{k}(b,h):=&  |\sum_j \sum_l \varepsilon_j \langle (I-P_{j-k+\kappa_{\delta}(k)})\Psi_{k,l} U^{k,l}_{j} (b_{j-k}), \bar{h}\rangle |;\\
 \mathcal{U}_2^{k}(b,h):=&  |\sum_j  \sum_l \varepsilon_j \langle (I-P_{j-k+\kappa_{\delta}(k)})(I-\Psi_{k,l}) U^{k,l}_{j} (b_{j-k}), \bar{h}\rangle |.
\end{align*}

We will estimate $\mathcal{U}_1^{k}(b,h)$ and $\mathcal{U}_2^{k}(b,h)$ separately.

\textbf{Estimate of $\mathcal{U}_1^{k}(b,h)$.}

For any fixed $k \geq 1$, it was proved in \cite[eq. (3.1)]{DL} that
\begin{equation}\label{Phi}
\sup_{\xi \neq 0} \sum_{l} | \Phi^2( 2^{k\alpha} \langle e^k_l, \xi' \rangle)| \lesssim 2^{k\alpha(n-2)}.
\end{equation}

By Plancherel's theorem and the Cauchy-Schwarz inequality, we obtain
\begin{equation}\label{PsiU}
\begin{aligned}
  &\|\sum_j \sum_l (I-P_{j-k+\kappa_{\delta}(k)})\Psi_{k,l} U^{k,l}_{j} (b_{j-k})\|_2^2\\
  &= \| \sum_l \Phi\big(2^{k\alpha} \langle e^k_l, \xi' \rangle\big) \mathcal{F}( \sum_j (I-P_{j-k+\kappa_{\delta}(k)})U^{k,l}_{j} (b_{j-k}))(\xi)\|_2^2\\
  &\leq  \|\sum_l \Phi^2\big(2^{k\alpha} \langle e^k_l, \xi' \rangle\big) \sum_l |\mathcal{F}( \sum_j (I-P_{j-k+\kappa_{\delta}(k)}) U^{k,l}_{j} (b_{j-k}))(\xi)|^2 \|_1\\
  &\lesssim  2^{k\alpha(n-2)} \sum_l \| \sum_j(I-P_{j-k+\kappa_{\delta}(k)}) U^{k,l}_{j} (b_{j-k}) \|_2^2.
\end{aligned}
\end{equation}

Recall the definition of $U^{k,l}_{j}$ in \eqref{Uklj}. By \eqref{Kkj}, we have
\begin{align*}
  |(I-P_{j-k+\kappa_{\delta}(k)})U^{k,l}_{j} (b_{j-k})(x)|
  \lesssim & |U^{k,l}_{j} (b_{j-k})(x)|+|P_{j-k+\kappa_{\delta}(k)}U^{k,l}_{j} (b_{j-k})(x)|\\
  \lesssim &  2^{-j n}\|\Omega_k\|_{L^{\infty}}\int_{2^{j-3}\leq |x-y|\leq 2^{j+1}}\Gamma^k_l(x-y)|b_{j-k}(y)|dy\\
   & + 2^{-j n}\|\Omega_k\|_{L^{\infty}} \int_{\mathbb{R}^n}\phi_{j-k+\kappa_{\delta}(k)}(x-w)\int_{2^{j-3}\leq |w-y|\leq 2^{j+1}}\Gamma^k_l(w-y)|b_{j-k}(y)|dydw\\
  \lesssim &  \|\Omega_k\|_{L^{\infty}} H^{k,l}_j \ast |b_{j-k}|(x),
\end{align*}
where $ H^{k,l}_j(x)=2^{-j n} \chi_{E^{k,l}_j}(x)$, and $\chi_{E^{k,l}_j}$ denotes the characteristic function of the set
$$
E^{k,l}_j:=\{x\in \mathbb{R}^n:  |  \langle x, e^k_l\rangle  | \leq  2^{j+2+\kappa_{\delta}(k)}, |x- \langle x, e^k_l\rangle e^k_l| \leq 2^{j+2-k\alpha+\kappa_{\delta}(k)}    \}.
$$
Therefore, for each fixed $e^k_l$, we have
\begin{align*}
   \| \sum_j (I-P_{j-k+\kappa_{\delta}(k)})U^{k,l}_{j} (b_{j-k})) \|_2^2 \lesssim &  \|\Omega_k\|_{L^{\infty}}^2 \sum_j \int_{\mathbb{R}^n} H^{k,l}_j\ast H^{k,l}_j\ast |b_{j-k}|(x) \cdot |b_{j-k}|(x) dx\\
  & + \|\Omega_k\|_{L^{\infty}}^2 \sum_j \sum_{i=-\infty}^{j-1} \int_{\mathbb{R}^n} H^{k,l}_j\ast H^{k,l}_i\ast |b_{j-k}|(x) \cdot |b_{j-k}|(x) dx.
\end{align*}

From the estimate $|E^{k,l}_j|\lesssim 2^{jn+\kappa_{\delta}(k)n-k\alpha (n-1)}$, it follows that
\[
\|H^{k,l}_i \|_1 \lesssim 2^{-in} |E^{k,l}_i| \lesssim 2^{\kappa_{\delta}(k)n-k\alpha(n-1)}.
\]
Consequently, for any \(i \leq j\), we have
\[
H^{k,l}_j * H^{k,l}_i(x) \lesssim 2^{\kappa_{\delta}(k)n-k\alpha(n-1)-jn} \chi_{\widetilde{E}^{k,l}_j}(x),
\]
where \(\widetilde{E}^{k,l}_j = E^{k,l}_j + E^{k,l}_j\).

Recall that for each $L\in \mathscr{Q}$, we have \(\|b_L\|_1 \lesssim  |L|\|b\|_{\dot{\mathcal{X}}_1}\), and the cubes $L$ are pairwise disjoint.
For fixed \(j\), \(k\), \(e^k_l\), we deduce that
\begin{align*}
&H^{k,l}_j * H^{k,l}_j * |b_{j-k}|(x) + 2 \sum_{i=-\infty}^{j-1} H^{k,l}_j * H^{k,l}_i * |b_{i-k}|(x)\\
&\lesssim 2^{\kappa_{\delta}(k)n-k\alpha(n-1)-jn } \sum_{i \leq j} \int_{x + \widetilde{E}^{k,l}_j} |b_{i-k}(y)|dy\\
&\lesssim 2^{\kappa_{\delta}(k)n-k\alpha(n-1)-jn}  \sum_{i \leq j} \sum_{\substack{L \in \mathscr{Q}, i_L=i-k \\ L \cap (x+\widetilde{E}^{k,l}_j) \neq \emptyset}} \int_{\mathbb{R}^n} |b_L(y)| \, dy\\
&\lesssim 2^{\kappa_{\delta}(k)n-k\alpha(n-1)-jn}  \sum_{i \leq j} \sum_{\substack{L \in \mathscr{Q}, i_L=i-k \\ L \cap (x+\widetilde{E}^{k,l}_j) \neq \emptyset}}  |L|\|b\|_{\dot{\mathcal{X}}_1}\\
&\lesssim 2^{\kappa_{\delta}(k)n-k\alpha(n-1)-jn} |\widetilde{E}^{k,l}_j|\|b\|_{\dot{\mathcal{X}}_1}\\
&\lesssim 2^{\kappa_{\delta}(k)n-k\alpha(n-1)-jn}  2^{\kappa_{\delta}(k)n+jn-k\alpha(n-1)} \|b\|_{\dot{\mathcal{X}}_1}\\
&= 2^{2\kappa_{\delta}(k)n-2k\alpha(n-1)}\|b\|_{\dot{\mathcal{X}}_1}.
\end{align*}

Moreover, using \(\|b_L\|_1 \lesssim  |L|\|b\|_{\dot{\mathcal{X}}_1}\) and the disjointness of $L$, we obtain
\begin{align*}
  \sum_{j} \int_{\mathbb{R}^n}|b_{j-k}(x)|dx \leq   \sum_{j} \sum_{L\in \mathscr{Q}, j_L=j-k } \int_{\mathbb{R}^n}|b_{L}(x)|dx
   \leq   \sum_{j} \sum_{L\in \mathscr{Q}, j_L=j-k } |L|\|b\|_{\mathcal{X}_1}
   \lesssim  |Q| \|b\|_{\dot{\mathcal{X}}_1}.
\end{align*}
Consequently,
\begin{align*}
   \| \sum_j(I-P_{j-k+\kappa_{\delta}(k)}) U^{k,l}_{j} (b_{j-k})) \|_2^2 \lesssim &  \|\Omega_k\|_{L^{\infty}}^2  2^{2\kappa_{\delta}(k)n-2k\alpha(n-1)}\|b\|_{\dot{\mathcal{X}}_1} \sum_{j} \int_{\mathbb{R}^n}|b_{j-k}(x)|dx\\
   \lesssim &  \|\Omega_k\|_{L^{\infty}}^2  2^{2\kappa_{\delta}(k)n-2k\alpha(n-1)} |Q|\|b\|_{\dot{\mathcal{X}}_1}^2.
\end{align*}

Substituting the above bound into \eqref{PsiU} and using card$(\Theta_k)\lesssim 2^{k\alpha (n-1)}$, we obtain
$$
\|\sum_j \sum_l (I-P_{j-k+\kappa_{\delta}(k)})\Psi_{k,l} U^{k,l}_{j} (b_{j-k})\|_2^2 \lesssim 2^{2\kappa_{\delta}(k)n-k\alpha} \|\Omega_k\|_{L^{\infty}}^2 |Q|\|b\|_{\dot{\mathcal{X}}_1}^2.
$$
By H\"{o}lder's inequality, and the fact that $\|h\|_2\lesssim |Q|^{1/2}\|h\|_{\mathcal{Y}_2}\lesssim |Q|^{1/2}\|h\|_{\mathcal{Y}_{\infty}}$, we have
\begin{equation}\label{U1bh}
\begin{aligned}
  \mathcal{U}_1^{k}(b,h)\leq &  \|\sum_j \sum_l(I-P_{j-k+\kappa_{\delta}(k)})\Psi_{k,l} U^{k,l}_{j} (b_{j-k})\|_2  \|h\|_2\\
  \lesssim & 2^{n \kappa_{\delta}(k)-k\alpha/2}  |Q|\|b\|_{\dot{\mathcal{X}}_1}  \|h\|_{\mathcal{Y}_{\infty}}\|\Omega_k\|_{L^{\infty}}.
\end{aligned}
\end{equation}

\textbf{Estimate of $\mathcal{U}_2^{k}(b,h)$.}

Before giving the estimate of $\mathcal{U}_2^{k}(b,h)$, we first introduce some notations, starting with the Littlewood-Paley decomposition.
Let \( \eta \in C^\infty\) be a radial function such that
\[
\eta(\xi) = 1 \text{ for } |\xi| \le 1, \quad
\eta(\xi) = 0 \text{ for } |\xi| \ge 2, \quad
0 \le \eta(\xi) \le 1 \text{ for all } \xi \in \mathbb{R}^n.
\]
Define
\[
\beta_m(\xi) = \eta(2^m \xi) - \eta(2^{m+1} \xi),
\]
so that \( \beta_m \) is supported in \( \{ \xi : 2^{-m-1} \leq |\xi| \leq 2^{-m+1} \} \).
Let \( \tilde{\beta}\in C^\infty\) be a radial function satisfying \[
\tilde{\beta}(\xi) = 1 \text{ for } \tfrac{1}{2} \le |\xi| \le 2,
\quad \supp \tilde{\beta} \subset \{ \xi : \tfrac{1}{4} \le |\xi| \le 4 \},
\quad 0 \le \tilde{\beta}(\xi) \le 1.
\]
Set $\tilde{\beta}_k(\xi)=\tilde{\beta}(2^m\xi)$.
Then it is easy to see $\beta_m=\tilde{\beta}_m\beta_m(\xi)$.
Define the convolution operators \( \Upsilon_m\) and \( \tilde{\Upsilon}_m \) with Fourier multipliers \( \beta_m \) and \( \tilde{\beta}_m \), respectively; that is,
\[
\widehat{\Upsilon_m f}(\xi) = \beta_m( \xi) \hat{f}(\xi), \quad \widehat{\tilde{\Upsilon}_m f}(\xi) = \tilde{\beta}_m(\xi) \hat{f}(\xi).
\]
It follows from the construction of \( \beta_m \) and \( \tilde{\beta}_m \) that $\Upsilon_m=\tilde{\Upsilon}_m \Upsilon_m, \sum_{m\in\mathbb{Z}}\Upsilon_m=I$.

We now turn to the estimate of $\mathcal{U}_2^{k}(b,h)$.
By H\"{o}lder's inequality, the triangle inequality, and the fact that $\|h\|_{\infty} \lesssim \|h\|_{\mathcal{Y}_{\infty}}$, we obtain
\begin{equation}\label{U2kbh}
  \begin{aligned}
    \mathcal{U}_2^{k}(b,h) \leq &   \|\sum_j  \sum_l (I-P_{j-k+\kappa_{\delta}(k)}) (I-\Psi_{k,l}) U^{k,l}_{j} (b_{j-k})\|_1 \|h\|_{\infty}\\
 \leq & \sum_j  \sum_l \sum_{L\in \mathscr{Q}, j_L=j-k} \| (I-P_{j-k+\kappa_{\delta}(k)}) (I-\Psi_{k,l}) U^{k,l}_{j} (b_{L})\|_1 \|h\|_{\mathcal{Y}_{\infty}}\\
 \leq & \sum_j  \sum_l \sum_{L\in \mathscr{Q}, j_L=j-k} \sum_m \| \Upsilon_m(I-P_{j-k+\kappa_{\delta}(k)}) (I-\Psi_{k,l}) U^{k,l}_{j} (b_{L})\|_1 \|h\|_{\mathcal{Y}_{\infty}}.
  \end{aligned}
\end{equation}
Therefore, it remains to estimate $ \| \Upsilon_m(I-P_{j-k+\kappa_{\delta}(k)}) (I-\Psi_{k,l}) U^{k,l}_{j} (b_{L})\|_1$.

Since $\Upsilon_m=\tilde{\Upsilon}_m \Upsilon_m$, it follows that
$$\Upsilon_m(I-P_{j-k+\kappa_{\delta}(k)}) (I-\Psi_{k,l}) U^{k,l}_{j}=\tilde{\Upsilon}_m \Upsilon_m(I-P_{j-k+\kappa_{\delta}(k)}) (I-\Psi_{k,l}) U^{k,l}_{j}.$$

To estimate this term, observe first that the support and cancellation properties of the kernel of $I-P_{j-k+\kappa_{\delta}(k)}$ imply (see also \cite{Seeger2})
\begin{equation}\label{xiaoshi}
\|\tilde{\Upsilon}_m (I-P_{j-k+\kappa_{\delta}(k)})\|_{L^1\to L^1}\lesssim \min\{1, 2^{j-k+\kappa_{\delta}(k)-m}\}.
\end{equation}

Next, note that
\[
\mathcal{F}[ (I - \Psi_{k,l}) \Upsilon_m f](\xi) = ( 1 - \Phi(2^{k\alpha}\langle e^k_l, \xi' \rangle) )\beta_m(\xi) \hat{f}(\xi).
\]
The multiplier above defines a bounded $L^1$ Fourier multiplier. Hence,
\begin{equation}\label{IPU}
\|(I - \Psi_{k,l}) \Upsilon_m\|_{L^1\to L^1} \lesssim 1,
\end{equation}
uniformly in $k, l$ and $m$ (see \cite[page 100]{Seeger} for details).

Recalling the definition of $U^{k,l}_{j}$ in \eqref{Uklj}, we obtain from \eqref{Kkj} and the support of $\Gamma^k_l$ that
\begin{align*}
  \|U^{k,l}_{j} f\|_1\leq 2^{-j n }\|\Gamma^k_l\chi_{\{2^{j-3}\leq|x|\leq 2^{j+1}\}}\|_1\|f\|_1\|\Omega_k\|_{L^{\infty}}
  \lesssim 2^{-k\alpha(n-1)}\|f\|_1\|\Omega_k\|_{L^{\infty}}.
\end{align*}

Together with \eqref{xiaoshi} and \eqref{IPU}, we obtain
\begin{equation}\label{msmall}
\begin{aligned}
  \|\Upsilon_m(I-P_{j-k+\kappa_{\delta}(k)}) (I-\Psi_{k,l}) U^{k,l}_{j}(b_L)\|_1\lesssim 2^{-k\alpha(n-1)} \min\{1, 2^{j-k+\kappa_{\delta}(k)-m}\}\|b_L\|_1\|\Omega_k\|_{L^{\infty}}.
\end{aligned}
\end{equation}

Note that the bound in \eqref{msmall} converges for small $m$, and we also need a separate estimate to handle large $m$.
Observe that $\|I-P_{j-k+\kappa_{\delta}(k)}\|_{L^1\to L^1}\lesssim 1.$
Therfore,
$$\|\Upsilon_m (I-P_{j-k+\kappa_{\delta}(k)}) (I-\Psi_{k,l}) U^{k,l}_{j} (b_{L})\|_1\leq \| \Upsilon_m (I-\Psi_{k,l}) U^{k,l}_{j} (b_{L})\|_1.$$

Let $R^{k,l}_m$ denote the Fourier multiplier of $\Upsilon_m (I-\Psi_{k,l}) $:
$$R^{k,l}_m(\xi)=(1-\Phi\big(2^{k\alpha} \langle e^k_l, \xi' \rangle\big))\beta_m(\xi).$$
By Fubini's theorem, we may write
\begin{align*}
   &\Upsilon_m (I-\Psi_{k,l}) U^{k,l}_{j} (b_{L})(x)\\
   &=\frac{1}{(2\pi)^n}\int_{\mathbb{R}^n} e^{i x\cdot \xi} R^{k,l}_m(\xi) \int_{\mathbb{R}^n} e^{-i w\cdot \xi} \int_{\mathbb{R}^n} \Gamma^k_l(w-y) \cdot \Omega_k(w-y) K^k_j(w,y) \cdot b_L(y) \, dy dw d\xi\\
   &=\frac{1}{(2\pi)^n}\int_{\mathbb{R}^n} \biggl\{\int_{\mathbb{R}^n} e^{i x\cdot \xi} R^{k,l}_m(\xi) \int_{\mathbb{R}^n} e^{-i w\cdot \xi}  \Gamma^k_l(w-y) \cdot \Omega_k(w-y) K^k_j(w,y) dw d\xi\biggr\} \cdot b_L(y) \, dy.
\end{align*}

Applying Minkowski's inequality, we obtain
\begin{equation}\label{UmIP}
\begin{aligned}
  &\|\Upsilon_m (I-\Psi_{k,l}) U^{k,l}_{j} (b_{L})\|_1\\
  &\leq \frac{1}{(2\pi)^n}\int_{\mathbb{R}^n} \Big\|\int_{\mathbb{R}^n} e^{i x\cdot \xi} R^{k,l}_m(\xi) \int_{\mathbb{R}^n} e^{-i w\cdot \xi}  \Gamma^k_l(w-y) \cdot \Omega_k(w-y) K^k_j(w,y) dw d\xi\Big\|_1 \cdot b_L(y) \, dy.
\end{aligned}
\end{equation}
Thus, it remains to estimate the $L^1$ norm of the inner integral.

Changing variables by $w=y+r\theta$, we have
\begin{align*}
  \int_{\mathbb{R}^n} e^{i x\cdot \xi} R^{k,l}_m(\xi) \int_{\mathbb{S}^{n-1}}\int_0^{\infty} e^{-i (y+r\theta)\cdot \xi}  \Gamma^k_l(\theta) \cdot \Omega_k(\theta) K^k_j(y+r\theta,y) r^{n-1}drd\theta d\xi\\
  =\int_{\mathbb{S}^{n-1}} \Omega_k(\theta) \Gamma^k_l(\theta)  \int_{\mathbb{R}^n} R^{k,l}_m(\xi)  \int_0^{\infty} e^{i (x-y-r\theta)\cdot \xi}  K^k_j(y+r\theta,y) r^{n-1}dr d\xi d\theta.
\end{align*}

Integrating \(N\) times by parts with respect to \(r\), the integral can be rewritten as
\[
\int_0^\infty e^{i \langle x - y - r \theta, \xi \rangle} (i \langle \theta, \xi \rangle)^{-N} \partial_r^{N} \left[ K_j^k(y + r \theta, y) r^{n-1} \right] dr.
\]

Observe that the operator $I-2^{-2m}\Delta_{\xi}$ is self-adjoint and acts on the exponential $e^{i\langle x-y-r \theta, \xi\rangle}$ as:
$$(I-2^{-2m}\Delta_{\xi})e^{i\langle x-y-r\theta, \xi\rangle}=(1+2^{-2m}|x-y-r \theta|^2)e^{i\langle x-y-r \theta, \xi\rangle}.$$
Therefore, the integral can be expresssed as
\begin{equation}\label{Part}\begin{aligned}
& \int_{\mathbb{S}^{n-1}} \Omega_k(\theta) \Gamma^k_l(\theta) \int_{\mathbb{R}^n} e^{i \langle x - y - r\theta, \xi \rangle} \int_0^\infty \partial_r^{N} \Big( K^k_j(y + r \theta, y) r^{n-1} \Big)  \\
&\times\frac{(I - 2^{-2m} \Delta_\xi)^{N_1}\Big( R^{k,l}_m (\xi) (i \langle \theta, \xi \rangle)^{-N} \Big)}{(1 + 2^{-2m} |x - y - r\theta|^2)^{N_1}}  \,dr \,d\xi  \,d\theta.
\end{aligned}\end{equation}

We now provide an explicit $L^1$ estimate for the term in \eqref{Part}.
From the definition of \(K^k_j(x, y)\) in \eqref{Kkj0}, it follows that
\begin{equation}\label{Part1}\begin{aligned}
|\partial_x^\sigma K^k_j(x, y)| &= 2^{-(j - \kappa_\delta(k))|\sigma|} \left| \int (\partial_x^\sigma \phi)_{j - \kappa_\delta(k)}(x - z) K_j(z, y)  dz \right| \\
&\leq 2^{-(j - \kappa_\delta(k))|\sigma|} \|K_j(\cdot, y)\|_\infty \|\partial_x^\sigma \phi\|_1 \\
&\lesssim 2^{-(j - \kappa_\delta(k))|\sigma| - jn}.
\end{aligned}\end{equation}

Applying the product rule, we obtain
\begin{align*}
| \partial_r^{N} \big( K^k_j(y + r \theta, y) r^{n-1} \big) |
&= | \sum_{i=0}^{N} C^i_{N} \partial_r^i (K^k_j(y + r \theta, y)) \partial_r^{N - i} (r^{n-1}) | \\
&= | \sum_{i = N - n + 1}^{N} C^i_{N} \partial_r^i (K^k_j(y + r \theta, y)) \partial_r^{N - i} (r^{n-1})|.
\end{align*}

By the support of \(K_j^k(x, y)\) in \eqref{Kkj}, we have \(2^{j-3} \leq r \leq 2^{j+1}\).
Using \eqref{Part1}, the above is bounded by
\begin{equation}\label{Part2}
| \partial_r^{N} \big( K^k_j(y + r \theta, y) r^{n-1} \big) |\leq \sum_{i = N - n + 1}^{N} C^i_{N} 2^{-(j - \kappa_\delta(k))i - jn} 2^{(j+1)(n-1 - N + i)} \leq C_{N} k^{2\delta^{-1} N} 2^{-j(1 + N)}.
\end{equation}

Moreover, it was shown in \cite[eq. (4.12)]{DL} that for any \(N_1 \in \mathbb{Z}_+\),
\begin{equation}\label{I-D}
\Big| (I - 2^{-2m} \Delta_\xi)^{N_1} \Big( \langle \theta, \xi \rangle^{-N} R^{k,l}_m(\xi) \Big) \Big| \leq C_{N} 2^{(k\alpha + m)N + 2k\alpha N_1}.
\end{equation}

We now choose \(N_1 = [n/2] + 1\). To establish the \(L^1\) estimate of \eqref{Part}, we use the support of \(R^{k,l}_m\), and Fubini's theorem to obtain
\begin{equation}\label{hksv}
\int_{\supp R^{k,l}_m} \int (1 + 2^{-2m} |x - y - r \theta|^2 )^{-N_1} \,dx  \,d\xi \lesssim 1.
\end{equation}

Combining \eqref{Part2}--\eqref{hksv}, we obtain the following bound for the $L^1$ norm of \eqref{Part}:
\begin{align*}
&C_{N} k^{2\delta^{-1} N} 2^{-j(1 + N) + (k\alpha + m)N + 2k\alpha N_1}\frac{1}{(2\pi)^n}
 \int_{\mathbb{S}^{n-1}} |\Omega_k(\theta) \Gamma^k_l(\theta) | \int^{2^{j+1}}_{2^{j-3}} \,dr \,d\theta\\
&\leq C_{N} k^{2\delta^{-1} N} 2^{-j(1 + N) + (k\alpha + m)N + 2k\alpha N_1+j-k\alpha(n-1)}\|\Omega_k\|_{L^{\infty}}\\
&=C_{N} k^{2\delta^{-1} N} 2^{(-j+m) N + k\alpha (N + 2 N_1) -k\alpha(n-1)}\|\Omega_k\|_{L^{\infty}}.
\end{align*}

Now, by \eqref{UmIP} and \eqref{Part}, we get
\begin{equation}\label{mlarge}
\begin{aligned}
  \| \Upsilon_m (I-\Psi_{k,l}) U^{k,l}_{j} (b_{L})\|_1\lesssim C_{N} k^{2\delta^{-1} N} 2^{(-j+m) N + k\alpha (N + 2 [n/2]+2) -k\alpha(n-1)}\|b_L\|_1\|\Omega_k\|_{L^{\infty}}.
\end{aligned}
\end{equation}

With \eqref{msmall} and \eqref{mlarge} at hand, we now complete the estimate of $ \mathcal{U}_2^{k}(b,h)$.
Let $m_0=j-[k/2]$ and choose $N=100n-n-2$. Then
\begin{align*}
  &\| (I-P_{j-k+\kappa_{\delta}(k)}) (I-\Psi_{k,l}) U^{k,l}_{j} (b_{L})\|_1\\
  &\leq \sum_m\| \Upsilon_m (I-P_{j-k+\kappa_{\delta}(k)})(I-\Psi_{k,l}) U^{k,l}_{j} (b_{L})\|_1\\
  &\lesssim \sum_{m\leq m_0}2^{-k\alpha(n-1)} \min\{1, 2^{j-k+\kappa_{\delta}(k)-m}\}\|b_L\|_1\|\Omega_k\|_{L^{\infty}}\\
  &\quad +\sum_{m> m_0}C_{N} k^{2\delta^{-1} N} 2^{(-j+m) N + k\alpha (N + 2 [n/2]+2) -k\alpha(n-1)}\|b_L\|_1\|\Omega_k\|_{L^{\infty}}\\
  &\lesssim 2^{-k\alpha(n-1)}(2^{[k/2]-k+\kappa_{\delta}(k)}+k^{200n\delta^{-1}} 2^{-[k/2] N + k\alpha (N + 2 [n/2]+2)} )  \|b_L\|_1\|\Omega_k\|_{L^{\infty}}\\
  &\lesssim 2^{-k\alpha(n-1)}(2^{-k/2+\kappa_{\delta}(k)}+k^{200n \delta^{-1}} 2^{-48 nk + 100 n k\alpha} ) \|b_L\|_1\|\Omega_k\|_{L^{\infty}}.
\end{align*}

Using \eqref{U2kbh} and the estimate above, we obtain
\begin{align*}
 \mathcal{U}_2^{k}(b,h)
 \leq & \sum_j  \sum_l \sum_{L\in \mathscr{Q}, j_L=j-k} \| (I-P_{j-k+\kappa_{\delta}(k)}) (I-\Psi_{k,l}) U^{k,l}_{j} (b_{L})\|_1 \|h\|_{\mathcal{Y}_{\infty}}\\
 \lesssim & \sum_j  \sum_l \sum_{L\in \mathscr{Q}, j_L=j-k}   2^{-k\alpha(n-1)}(2^{-k/2+\kappa_{\delta}(k)}+k^{200n\delta^{-1}} 2^{-48 n k+ 100 n k\alpha} ) \|b_L\|_1 \|h\|_{\mathcal{Y}_{\infty}}\|\Omega_k\|_{L^{\infty}}\\
  \lesssim & (2^{-k/2+\kappa_{\delta}(k)}+k^{200n\delta^{-1}} 2^{-48 nk + 100 n k\alpha} ) \|b\|_{\dot{\mathcal{X}}_1} \|h\|_{\mathcal{Y}_{\infty}}\|\Omega_k\|_{L^{\infty}},
\end{align*}
where the last inequality follows from \(\|b_L\|_1 \lesssim  |L|\|b\|_{\dot{\mathcal{X}}_1}\), the disjointness of $L$, and card$(\Theta_k)\lesssim 2^{k\alpha (n-1)}$.

Taking $\alpha=1/4$, and combining the estimates of $\mathcal{U}_1^{k}(b,h)$ in \eqref{U1bh} and $\mathcal{U}_2^{k}(b,h)$ above, we obtain
\begin{align*}
  \mathcal{U}^{k}(b,h)\leq \mathcal{U}_1^{k}(b,h)+\mathcal{U}_2^{k}(b,h)\lesssim k^{200n\delta^{-1}}  2^{-k/8} |Q|\|b\|_{\dot{\mathcal{X}}_1}  \|h\|_{\mathcal{Y}_{\infty}}\|\Omega_k\|_{L^{\infty}}.
\end{align*}
This completes the proof of Lemma \ref{Ubh}.
\end{proof}

\section{Proofs of the Application Theorems}\label{sec-pfTb}
\setcounter{equation}{0}

In this section, we provide the proofs of the results stated in the Applications subsection of the Introduction.

\begin{proof}[Proof of Theorem \ref{wTbk}]
To apply Corollary \ref{cor1}, we need to verify two facts:

(i) $T_{\Omega, B, k}$ is bounded on $L^{2}$;

(ii) its kernel satisfies \eqref{Ksize} and \eqref{Ksmooth}.

The $L^{2}$ boundedness
\[
\|T_{\Omega, B, k}f\|_{2}
   \lesssim \|\nabla B\|_\infty^{k}\|\Omega\|_{L^{q}}\|f\|_{2}
\]
is a direct consequence of \cite[Theorem 1.3a]{Hof90};
see also \cite{CHT24} for extensions to Grafakos and Stefanov \cite{GraS} type kernels.
Thus it remains to check that the kernel
\[
K(x,y)= \frac{1}{|x-y|^{n}} \Big(\frac{B(x)-B(y)}{|x-y|}\Big)^{k}
\]
satisfies the size condition \eqref{Ksize} and the smoothness condition \eqref{Ksmooth}.

As proved in \cite[Subsection 5.2]{DL}, we have
$$
  |K(x,y)| \lesssim \|\nabla B\|^{k}_{\infty}  \frac{1}{|x-y|^n},
$$
and
$$
\begin{aligned}
|K(x_1,y)-K(x_2,y)|\lesssim &  \|\nabla B\|^{k}_{\infty}   \frac{|x_1-x_2|}{|x_1-y|^{n+1}},\quad |x_1-y|>2|x_1-x_2|,\\
|K(x,y_1)-K(x,y_2)|\lesssim & \|\nabla B\|^{k}_{\infty}   \frac{|y_1-y_2|}{|x-y_1|^{n+1}},\quad |x-y_1|>2|y_1-y_2|.
\end{aligned}
$$
This completes the proof.
\end{proof}

\begin{proof}[Proof of Theorem \ref{Tbr}]
It was shown in \cite[Theorem 1.1]{CDH16} that
 \begin{equation*}
    \|T_{\Omega,B,\gamma}f\|_{2}\lesssim \|D^{\gamma}B\|_{\mathrm{BMO}}\|\Omega\|_{L^q} \|f\|_{2}.
  \end{equation*}
By Corollary \ref{cor1}, it suffices to verify that the kernel
$$K(x,y)=\frac{1}{|x-y|^{n}} \frac{B(x)-B(y)}{|x-y|^{\gamma}}$$
satisfies \eqref{Ksize} and \eqref{Ksmooth}.

It was shown by Strichartz \cite{Stri} that if $\gamma \in (0,1)$,
the space $I_\gamma(\mathrm{BMO})(\mathbb{R}^n)$, considered modulo constants,
is properly contained in the Lipschitz space
$\mathrm{Lip}_\gamma(\mathbb{R}^n)$,
consisting of functions $f$ satisfying
$|f(x)-f(y)| \le C|x-y|^\gamma$ for all $x,y \in \mathbb{R}^n$. In particular,
$$\|B\|_{\mathrm{Lip}_{\gamma}}\lesssim \|D^{\gamma}B\|_{\mathrm{BMO}}.$$
Hence,
$$
  |K(x,y)| \lesssim \|D^{\gamma}B\|_{\mathrm{BMO}} \frac{1}{|x-y|^n},
$$
and by the mean value theorem, for $|x_1-y|>2|x_1-x_2|$, we have
\begin{align*}
|K(x_1,y)-K(x_2,y)|
\leq & \frac{| B(x_1)-B(x_2)|}{|x_1-y|^{n+\gamma}} + |B(x_2)-B(y)| \Big| \frac{1}{|x_1-y|^{n+\gamma}}-\frac{1}{|x_2-y|^{n+\gamma}} \Big| \\
\lesssim &  \|D^{\gamma}B\|_{\mathrm{BMO}}  \frac{|x_1-x_2|^{\gamma}}{|x_1-y|^{n+\gamma}},
\end{align*}
and similarly, for $|x-y_1|>2|y_1-y_2|$, we have
$$
|K(x,y_1)-K(x,y_2)|\lesssim  \|D^{\gamma}B\|_{\mathrm{BMO}}  \frac{|y_1-y_2|^{\gamma}}{|x-y_1|^{n+\gamma}}.
$$
Therefore, the kernel $K$ satisfies the required size and smoothness conditions,
and the proof is complete.
\end{proof}

\begin{proof}[Proof of Theorem \ref{bF1}]
Under the assumptions of Theorem \ref{bF1}, Calder\'{o}n \cite{Ca80} proved the $L^2$ boundedness
$$\|T_{\Omega, B, F}\|_2 \lesssim C_{B,F} \|\Omega\|_{L^1} \|f\|_2 ,$$
where $C_{B,F}$ depends only on $n$, $B$, and $F$.
Thus by Corollary \ref{cor1}, it suffices to verify that the kernel
$$K(x,y)=\frac{1}{|x-y|^{n}}F\Big(\frac{B(x)-B(y)}{|x-y|}\Big)$$
satisfies \eqref{Ksize} and \eqref{Ksmooth}.
According to \cite[Subsection 5.3]{DL}, we see that
$$
  |K(x,y)| \lesssim C_{B,F}\frac{1}{|x-y|^n},
$$
and
$$
\begin{aligned}
|K(x_1,y)-K(x_2,y)|\lesssim &  C_{B,F} \frac{|x_1-x_2|}{|x_1-y|^{n+1}},\quad |x_1-y|>2|x_1-x_2|,\\
|K(x,y_1)-K(x,y_2)|\lesssim & C_{B,F} \frac{|y_1-y_2|}{|x-y_1|^{n+1}},\quad |x-y_1|>2|y_1-y_2|.
\end{aligned}
$$
Hence, we complete the proof.
\end{proof}

\begin{proof}[Proof of Theorem \ref{Tbpl}]
Under the assumptions of Theorem \ref{Tbpl}, Baj\v{s}anski and Coifman \cite{Bajsan} showed that
$$\|T_{\Omega, B}^l f\|_2 \lesssim \sum_{|\alpha|=l} \|B_{\alpha}\|_{\infty} \|\Omega\|_{L^q}\|f\|_2 .$$
Therefore, by Corollary \ref{cor1}, it suffices to verify that the kernel
$$K(x,y)=\frac{1}{|x-y|^{n}}\frac{P_l(B,x,y)}{|x-y|^l}$$
satisfies \eqref{Ksize} and \eqref{Ksmooth}.
As shown in \cite[Subsection 5.4]{DL}, we have
$$
  |K(x,y)| \lesssim \sum_{|\alpha|=l} \|B_{\alpha}\|_{\infty}\frac{1}{|x-y|^n},
$$
and
$$
\begin{aligned}
|K(x_1,y)-K(x_2,y)|\lesssim & \sum_{|\alpha|=l} \|B_{\alpha}\|_{\infty} \frac{|x_1-x_2|}{|x_1-y|^{n+1}},\quad |x_1-y|>2|x_1-x_2|,\\
|K(x,y_1)-K(x,y_2)|\lesssim & \sum_{|\alpha|=l} \|B_{\alpha}\|_{\infty} \frac{|y_1-y_2|}{|x-y_1|^{n+1}},\quad |x-y_1|>2|y_1-y_2|.
\end{aligned}
$$
This completes the proof.
\end{proof}

\begin{proof}[Proof of Theorem \ref{ir1}]
Under the assumptions of Theorem \ref{ir1}, it was shown in \cite[Theorem 8]{Mucken} that
$$\|T_{\Omega, ir}f\|_2\lesssim \|\Omega\|_{L^1}\|f\|_2 .$$
Thus by Corollary \ref{cor1}, it suffices to verify that the kernel
$$K(x,y)=\frac{1}{|x-y|^{n+ir}}$$
satisfies \eqref{Ksize} and \eqref{Ksmooth}.
Note that
$$
  |K(x,y)| = \frac{1}{|x-y|^n}.
$$
Using the mean value theorem, we also have
$$
\begin{aligned}
|K(x_1,y)-K(x_2,y)|\lesssim & \frac{|x_1-x_2|}{|x_1-y|^{n+1}},\quad |x_1-y|>2|x_1-x_2|,\\
|K(x,y_1)-K(x,y_2)|\lesssim & \frac{|y_1-y_2|}{|x-y_1|^{n+1}},\quad |x-y_1|>2|y_1-y_2|.
\end{aligned}
$$
The proof is complete.
\end{proof}

%
%
%
%
%
%
%

\bigskip
\noindent
{\bf Declaration of competing interest}

The authors declare that there is no conflict of interest.

\bigskip
\noindent
{\bf Data availability}

No data was used for the research described in the article.


\bigskip
\section*{ Acknowledgements} 
This work was in part supported by: Yanping Chen's National Natural Science Foundation of China [grant numbers 12525105, 12371092 ].

\end{document}